\pdfoutput=1
\documentclass[a4paper]{amsart}

\makeatletter                             
\g@addto@macro\@floatboxreset\centering   
\makeatother                              

\usepackage[utf8]{inputenc}
\usepackage{amsthm}
\usepackage{amsfonts}
\usepackage{amsmath}
\usepackage{amssymb}
\usepackage{stmaryrd}
\usepackage[all,cmtip]{xy}
\usepackage{enumitem}
\usepackage{tikz}
\usepackage{tikz-cd}
\usepackage{scalefnt}
\usepackage{amscd}
\usepackage[margin=1cm]{caption}
\usepackage{microtype}

\usetikzlibrary{shapes.geometric}

\numberwithin{equation}{section}

\newtheorem{theorem}{Theorem}[section]
\newtheorem{lemma}[theorem]{Lemma}
\newtheorem{proposition}[theorem]{Proposition}
\newtheorem{cor}[theorem]{Corollary}

\newtheorem*{maint}{Theorem \ref{maint}}
\newtheorem*{negat}{Theorem \ref{negat}}

\theoremstyle{definition}
\newtheorem{definition}[theorem]{Definition}

\theoremstyle{remark}
\newtheorem{remark}[theorem]{Remark}

\def\K{\ensuremath\mathrm{K}}
\def\LCX{\ensuremath\mathbb{LC}(X)}
\def\R{\ensuremath\mathcal{R}}
\def\T{\ensuremath\mathfrak{T}}

\def\Z{\ensuremath\mathbb{Z}}

\DeclareMathOperator{\im}{im}
\DeclareMathOperator{\supp}{supp}
\DeclareMathOperator{\Spec}{Spec}
\DeclareMathOperator{\holim}{{ho}{-}{\varinjlim}}
\DeclareMathOperator{\coker}{coker}
\DeclareMathOperator{\Prim}{Prim}
\DeclareMathOperator{\Ch}{Ch}

\title[localizing subcategories of  $\mathcal{B}(X)$]{Localizing subcategories in the bootstrap category of filtered C*-algebras}

\author{George Nadareishvili}

\begin{document}

\begin{abstract}

We use the abelian approximation for the bootstrap category of filtered C*-algebras to define a sensible notion of support for its objects. As a consequence, we provide a full classification of localizing subcategories in terms of a product of lattices of noncrossing partitions of a regular $(n+1)$-gon, where $n$ is the number of ideals in the filtration.
  
\end{abstract}
\maketitle

\section{introduction}

A full subcategory of a triangulated category is called \emph{localizing} if it is closed under suspension, formation of triangles and whatever coproducts exist in the ambient category. A general classification program aims to establish a lattice isomorphism between such localizing subcategories and a suitable computable lattice. As demonstrated for example in \cite{HS98} or \cite{Nee92}, such a classification can be used to obtain an interesting invariant or different structural information about a triangulated category. 

The classification of localizing subcategories always proceeds by defining a notion of \emph{support} for objects in a triangulated category. This is usually a canonical process of assigning a subset of a certain space to every object. Generally speaking, the support introduces a geometric approach for studying an algebraic structure. 

Once we have a good definition of support, the classification result should say, first, that any subcategory is determined uniquely by the supports of its objects; secondly, it should describe the sets that appear as supports of localizing subcategories.

If one works with a compactly generated triangulated category with small coproducts and with an action of a commutative noetherian ring $R$,  Benson-Iyengar-Krause~\cite{BIK08} define supports based on a construction of local cohomology functors with respect to the ring $R$. Then $\Spec R$ naturally serves as a locus for supports. This method is rather powerful, and classifications like~\cite{Nee92}, \cite{BIK11} and a few others fall under this theory. However, the triangulated categories we are interested in are not compactly generated in the usual sense because they do not have arbitrary small coproducts. This obstruction is not trivial, since a very crucial fact used by Benson-Iyengar-Krause, namely the classical Brown representability, does not hold. In addition, in our case, any ring $R$ that acts on a category and is large enough to accommodate a sensible notion of support is noncommutative, and therefore there is no good candidate for $\Spec R$. To explain our approach we first describe the setup.

A C*-algebra over a topological space $X$, shortly an $X$-C*-algebra, is a pair $(A, \psi)$, where $A$ is a C*-algebra and $\psi: \Prim(A) \rightarrow X$ a continuous map.  $\mathfrak{KK}(X)$ is defined to be the Kasparov category of C*-algebras over $X$: its objects are separable C*-algebras over $X$, its morphism set from $A$ to $B$ is $\mathrm{KK}_0(X;A,B)$: an $X$-equivariant version of Kasparov's bivariant K-group in degree zero~\cite{MN09}. The composition is given by the corresponding Kasparov product.

As demonstrated by Meyer-Nest~\cite{MN09}, $\mathfrak{KK}(X)$ is a triangulated category.

From now on, unless stated otherwise, assume $X$ to be finite, $T_0$, with totally ordered lattice of open subsets. Let $n=|X|$. Then a C*-algebra over~$X$ is equivalent to a C*-algebra with an increasing chain of ideals $$I_n \triangleleft I_{n-1} \triangleleft \cdots  \triangleleft I_2 \triangleleft  I_1 =A.$$ We briefly call this a filtered C*-algebra. For $n=2$ this is an extension of C*-algebras.

The \emph{bootstrap category} $\mathcal{B}(X)$ is defined as the smallest localizing subcategory of $\mathfrak{KK}(X)$ which contains all the possible ways $\mathbb{C}$ can be made into a C*-algebra over $X$~\cite{MN09}.
Another description of the bootstrap subcategory was also derived by Meyer-Nest~\cite{MN12}: a C*-algebra over $X$ belongs to  $\mathcal{B}(X)$ if and only if it satisfies an appropriate Universal Coefficient Theorem. This will be recalled in Section~\ref{prelim}.

\subsection*{Classification for the bootstrap class.} 

If the space $X$ is just a single point, we recover the original definitions of Kasparov's $\mathfrak{KK}$ category and a bootstrap class~$\mathcal{B}$, characterized by the classical Universal Coefficient Theorem by Rosenberg and Schochet~\cite{RS87}. Brown representabiliy already fails here, but we still have the action of a commutative noetherian endomorphism ring $\mathbb{Z} \cong \mathrm{End}_{\mathcal{B}}(\mathbb{C})$ of the tensor unit object $\mathbb{C} \in \mathcal{B}$. 

Ivo Dell'Ambrogio~\cite{Del11} classified localizing subcategories of the bootstrap class~$\mathcal{B}$ in terms of subsets of the spectrum  of this ring of integers. As in \cite{Nee92} and \cite{BIK08}, to define supports for objects in $\mathcal{B}$ Dell'Ambrogio uses the collection of functors $\mathbf{C}_{\mathcal{B}}=\{\mathrm{K}_*(-;\mathbb{F}_p) \mid p \in \Spec \mathbb{Z} \}$, where $\mathrm{K}_*(-;\mathbb{F}_p)$ is the $\mathrm{K}$-theory with coefficients in the residue field $\mathbb{F}_p$; that is $\mathbb{F}_p=\mathbb{Z}/p$ for $p\neq0$ and $\mathbb{F}_p=\mathbb{Q}$ for $p=0.$ More precisely, the support of the object $A \in \mathcal{B}$ is the subset of $\Spec \mathbb{Z}$ for which the corresponding functors in~$\mathbf{C}_{\mathcal{B}}$ do not vanish on $A$. We are going to generalize this classification result to $\mathcal{B}(X)$, where there is no action of a large enough commutative ring.

\subsection*{Classification for $\mathcal{B}(X)$.}

Let $Y \subseteq X$ be a \emph{locally closed subset}; that is, a subset that is a difference of two open sets in $X$. There are exactly $m=\frac{n(n+1)}{2}$ locally closed subsets in $X$, namely, the intervals $[a,b]$ for $1 \leq a\leq b \leq n$. For each such $Y$, there is a homological functor $\mathrm{FK}_Y: \mathcal{B}(X) \rightarrow \mathfrak{Ab}^{\mathbb{Z}/2}$ into the category of ${\mathbb{Z}/2}$-graded abelian groups, which computes the $\mathrm{K}$-theory of a subquotient corresponding to $Y$. We choose our collection to be $$\mathbf{C}_{\mathcal{B}(X)}=\{\mathrm{FK}_Y(-;\mathbb{F}_p) \mid p \in \Spec \mathbb{Z},\,Y \text{ is locally closed} \},$$ and define $$\supp A:=\{ (p,Y) \mid \mathrm{FK}_Y(A;\mathbb{F}_p)\neq 0 \}.$$ The support of a localizing subcategory is defined to be the union of the supports of its objects. 

This way, the supports of objects in $\mathcal{B}(X)$ live in an $m$-fold cartesian product of power sets of $\Spec \mathbb{Z}$. However, unlike in the commutative case, not all elements of this product appear as supports of some localizing subcategory. There is a dependence between functors in $\mathbf{C}_{\mathcal{B}(X)}$: for any fixed $p \in \Spec \mathbb{Z}$ there are exactly $\frac{1}{n+2}\binom {2n+2} {n+1}$ ($(n+1)$th Catalan number) different localizing subcategories with $p$ as a first coordinate in every support point.
These subcategories form a lattice isomorphic to the lattice of \emph{noncrossing partitions} -- those partitions of a regular $(n+1)$-gon which do not cross in their planar representation. Summing up, this leads to our main result:

\begin{maint}
The lattice of all localizing subcategories of $\mathcal{B}(X)$ is isomorphic to the product of lattices of noncrossing partitions of the regular $(n+1)$-gon over the indexing set $\Spec \mathbb{Z}$.
\end{maint}

In order to better illustrate what this classification says, consider the example where $X$ has only two points. The category of C*-algebras over $X$ is equivalent to the $\mathfrak{KK}$-category of extensions of C*-algebras. Then the theorem classifies all localizing subcategories of the bootstrap class in the $\mathfrak{KK}$-category of extensions of C*-algebras in terms of those triples of subsets of $\Spec \mathbb{Z}$ which have the property that each one is inside the union of the other two. This is not unexpected, since an earlier result by Alexander Bonkat~\cite{AB02} establishes that isomorphism classes of objects in the bootstrap class of C*-algebra extensions correspond to isomorphism classes of 6-periodic exact chain complexes of countable abelian groups.

This example already reveals how our classification is different from the commutative case. Unfortunately, we cannot hope to recover a space from the lattice of localizing subcategories:

\begin{negat}
The lattice of localizing subcategories of the bootstrap category $\mathcal{B}(X)$ is not isomorphic to a sublattice of a subset lattice $\mathcal{P}(S)$ for any set $S$.
\end{negat}

This already fails for the lattice of noncrossing partitions of the triangle. We will fill in the details for this example later in Section~\ref{Localizing subcategories}.

However, first we start by providing some preliminaries and fixing some notation in Section~\ref{prelim}.

In Section~\ref{Properties of boot}, we prove some general results about localizing subcategories of  $\mathcal{B}(X)$, for arbitrary $X$. Namely, we show that localizing subcategories are closed under tensoring with C*-algebras, and that they are generated by localization of C*-algebras over $X$ at prime numbers and zero. These results are used in Section~\ref{Localizing subcategories} to prove the preliminary classification theorem, which classifies localizing subcategories in terms of certain elements of the $m$-fold cartesian product of power sets of $\mathrm{Spec}\,\mathbb{Z}.$

In Section~\ref{Noncorssing partitions}, we recall the classical definition of the lattice of noncrossing partitions. Then, using the mentioned preliminary theorem, we prove Theorem \ref{maint}.

\subsection*{Acknowledgments} I would like to express my gratitude to my advisor Ralf Meyer for scientific guidance and many fruitful discussions. I would also like to thank Thomas Schick and Ivo Dell'Ambrogio for useful suggestions on my work.

\section{Preliminaries} \label{prelim}

In this section we will recall some known results about the Kasparov category of C*-algebras over a topological space which will be used later on. Unless indicated otherwise, all these results can be found in~\cite{MN12}.

If $X$ is a finite topological space and $T_0$, giving a continuous map $\psi : \Prim(A) \rightarrow X$, which gives to a C*-algebra $A$ a structure of
C*-algebra over $X$, is equivalent to giving a map $$\psi^*:\mathbb{O}(X)\rightarrow \mathbb{I}(A), \qquad U \mapsto \psi^{-1}(U)=:A(U),$$ 
that preserves finite infima and arbitrary suprema. Here $\mathbb{O}(X)$ and $\mathbb{I}(A)$ denote the lattices of open subsets of $X$ and closed *-ideals in $A$, respectively.
In particular, it follows that $A(U) \triangleleft A(V)$ for $ U \subseteq V$ and that $A(\emptyset)=\{0\},\, A(X)=A$. 

A *-homomorphism $f:A \rightarrow B$ between two C*-algebras over $X$ is \emph{$X$-equivariant} if $f(A(U)) \subseteq B(U)$ for every open set $U \subseteq X$.

Since $X$-equivariant maps only use $\mathbb{O}(X)$ in their definition, the category of $X$-C*-algebras is equivalent to the category of $\hat{X}$-C*-algebras, where $\hat{X}$ is the $T_0$ completion of $X$. So for our purposes we might as well assume that $X$ is $T_0$.

Therefore, if $X$ has totally ordered lattice of open subsets, it can be identified with the totally ordered set $\{1,\dots,n\}$ with the \emph{Alexandrov topology}: the topology where a set is open if and only if it is of the form $[a,n]:=\{x \in X \mid a\leq x\} $ for some $a \in X$. 

Even though the constructions and some of the theorems in this section work for more general spaces, from now on we assume that $X$ has totally ordered lattice of open subsets and we identify $X=\{1,\dots,n\}$; then a C*-algebra over $X$ is the same as a C*-algebra with a finite filtration by ideals $$I_n \triangleleft I_{n-1} \triangleleft \cdots  \triangleleft I_2 \triangleleft  I_1 =A,$$ where each $I_i=A([i,n])$ for $[i,n] \in \mathbb{O}(X)$.

A subset $Y\subseteq X$ is called locally closed if $Y=U \setminus V$ for some $U,V \in \mathbb{O}(X)$ with $V \subseteq U$. If $Y$ is non-empty, this exactly corresponds to the intervals in $X=\{1,\dots,n\}$; that is $Y=[a,n]\setminus [b+1,n]=[a,b]:=\{x \in X \mid a\leq x \leq b \}$ for some $1\leq a \leq b\leq n$. We denote the set of all non-empty locally closed sets, or equivalently intervals in $X$, by $\LCX$. For an $X$-C*-algebra $A$ and $[a,b]=[a,n]\setminus [b+1,n]$ as above, define $A([a,b]):=A([a,n])/A([b+1,n])$.

As mentioned in the introduction, $\mathfrak{KK}(X)$ is a triangulated category. The suspension functor is $\Sigma:=\mathrm{C}_0(\mathbb{R},-) $. The triangles come from semi-split C*-algebra extensions over $X$. In other words, every triangle $\Sigma Q \rightarrow I\rightarrow E \rightarrow Q $ in $\mathfrak{KK}$ corresponds to the set of extensions of C*-algebras $I(U) \hookrightarrow E(U) \twoheadrightarrow Q(U),\,U \in \mathbb{O}(X)$, with a completely positive, contractive section $Q \rightarrow E$ which restricts to sections $Q(U) \rightarrow E(U)$ for all $U \in \mathbb{O}(X)$ \cite{MN09}.

For every $Y \in \LCX$, we define a functor 
$$\mathrm{FK}_Y: \mathfrak{KK}(X) \rightarrow \mathfrak{Ab}^{\mathbb{Z}/2}, \qquad \mathrm{FK}_Y(A):=\K_*(A(Y)).$$
Here $\mathfrak{Ab}^{\mathbb{Z}/2}$ denotes the category of $\mathbb{Z}/2$-graded abelian groups. Meyer and Nest combine the functors $\mathrm{FK}_Y$ for all $Y \in \LCX$ into a single \emph{filtrated $\K$-theory} functor. The latter, however, also includes its target category, which we recall below.

Let  $\mathcal{NT}$ be the small, $\mathbb{Z}/2$-graded, pre-additive category with object set $\LCX$, and the $\mathbb{Z}/2$-graded abelian group of natural transformations $\mathrm{FK}_Y \Rightarrow \mathrm{FK}_Z$ as arrows $Y \rightarrow Z$.

Let $\mathfrak{Mod}(\mathcal{NT})$ be the abelian category of grading preserving, additive functors $\mathcal{NT} \rightarrow \mathfrak{Ab}^{\mathbb{Z}/2}$. These functors are usually called modules, hence the notation. Then filtrated K-theory is the functor $$\mathrm{FK}=(\mathrm{FK}_Y)_{Y \in \LCX}:\mathfrak{KK}(X) \rightarrow \mathfrak{Mod}(\mathcal{NT})_c, \qquad A \mapsto \Big(\K_*(A(Y))\Big)_{Y \in \LCX}.$$ Here $\mathfrak{Mod}(\mathcal{NT})_c$ denotes the full subcategory of countable modules in $\mathfrak{Mod}(\mathcal{NT})$.

Let $\mathrm{Hom}_{\mathcal{NT}}$ and $\mathrm{Ext}^1_{\mathcal{NT}}$ denote the morphism and extension groups in the abelian category $\mathfrak{Mod}(\mathcal{NT})_c$, respectively. Then the following universal coefficient theorem holds:

\begin{theorem}[Meyer-Nest \cite{MN12}]
For any $A \in \mathcal{B}(X)$ and $B \in \mathfrak{KK}(X)$, there are natural short exact sequences
$$\mathrm{Ext}^1_{\mathcal{NT}}\big(\mathrm{FK}(A)[j+1],\mathrm{FK}(B)\big)\hookrightarrow \mathrm{KK}_{j}(X;A,B) \twoheadrightarrow \mathrm{Hom}_{\mathcal{NT}}\big(\mathrm{FK}(A)[j],\mathrm{FK}(B)\big)$$ for $j \in \mathbb{Z}/2$, where $[j]$ and $[j+1]$ denote degree shifts.
\end{theorem}

\begin{cor}[Meyer-Nest \cite{MN12}] \label{uct-cor}
Let $M \in \mathfrak{Mod}(\mathcal{NT})_c$ have a projective resolution of length $1$. Then there is $A \in \mathcal{B}(X)$ with $\mathrm{FK}(A) \cong M$, and this object is unique up to isomorphism in $\mathcal{B}(X)$.
\end{cor}

In order to describe generators for the localizing subcategories in $\mathcal{B}(X)$, we will extensively use the special $X$-C*-algebras $\R_Y \in \mathcal{B}(X)$ for $ Y \in \LCX$. These objects are characterized by the following representability theorem:

\begin{theorem}[Meyer-Nest \cite{MN12}] \label{repres}
The covariant functors $\mathrm{FK}_Y$ for $Y \in \LCX$ are representable, that is, there are objects $\R_Y \in \mathfrak{KK}(X)$ and natural isomorphisms $$\mathrm{KK}_*(X;\R_Y,A) \cong \mathrm{FK}_Y(A)=\K_*\big(A(Y)\big)$$ for all $A \in \mathfrak{KK}(X),\, Y \in \LCX$.
\end{theorem}

We also recall the explicit description of the objects  $\R_Y$ for $ Y \in \LCX$. Let $\Ch(X)$ be a closed simplex of dimension $n-1$, whose $m$-simplices are the chains $x_1 \leq x_2 \leq \dots \leq x_m$, for $m \leq n-1$ and whose face and degeneracy maps delete or double an entry in the chain. Let $\Delta_{[a,b]}$ denote the $b-a$-dimensional face of $\Ch(X)$ corresponding to $[a,b]$. Let $\Delta_{[a,b]}^{\mathrm{o}} := \Delta_{[a,b]} \setminus \partial \Delta_{[a,b]}$ be the open simplex. Moreover, let $X^{\mathrm{op}}$ be $X$ with the topology of the reversed partial order $\geq$. Then~\cite[Proposition~2.8]{MN12} gives a continuous map 
$$(m,M): \Ch(X) \rightarrow X^{\mathrm{op}} \times X,$$
where for $x \in \Delta_{[a,b]}^{\mathrm{o}}$, we define $m(x)=a$ and $M(x)=b$. 

Let $\R:=\mathrm{C}\big(\Ch(X)\big)$, the C*-algebra of continuous functions on $\Ch(X)$. Since $$ \Prim\R=\Prim\mathrm{C}\big(\Ch(X)\big) \cong \Ch(X),$$ the map $(m,M)$ turns $\R$ into a C*-algebra over $X^{\mathrm{op}} \times X.$

Define $\R_Y$ to be the restriction of $\R$ to $Y^{\mathrm{op}}\times X$, viewed as an $X$-C*-algebra via the coordinate projection  $Y^{\mathrm{op}}\times X \rightarrow X$, where $Y^{\mathrm{op}}$ is $Y$ with the subspace topology coming from $X^{\mathrm{op}}.$ In other words, $$\R_Y(Z):=\R(Y^{\mathrm{op}}\times Z)=\mathrm{C}_0\big(m^{-1}(Y) \cap M^{-1}(Z)\big).$$

Theorem \ref{repres} together with the Yoneda Lemma gives 
\begin{align*}
\mathcal{NT}_*(Y,Z) \cong \mathrm{KK}_*(X;\R_Z,\R_Y) \cong \mathrm{FK}_Z(\R_Y)=  \K_*\big(\R_Y(Z)\big) & =\K_*\big(\R(Y^{\mathrm{op}}\times Z)\big) \\ 
=\K^*\big(m^{-1}(Y) \cap M^{-1}(Z)\big).
\end{align*}
Therefore, computing $\K_*\big(\R_Y(Z)\big) \cong \mathcal{NT}_*(Y,Z)$ comes down to computing the topological $\K$-theory of some simplices. After doing so for $Y=[a,b]$ and $Z=[c,d]$, one gets \cite[Section~3.1]{MN12}
\begin{equation} \label{conditions}
    \K_*\big(\R_{[a,b]}([c,d])\big)\cong 
\begin{cases}
    \mathbb{Z}[0]& \text{if } c \leq a \leq d \leq b,\\
    \mathbb{Z}[1]& \text{if } a< c,\, b<d \text{ and } c-1 \leq b,\\            
    0  & \text{otherwise.}
\end{cases}
\end{equation}

By the classical Universal Coefficient Theorem, the $\K$-theory functor is a complete invariant for the bootstrap class $\mathcal{B}$. So since $\K_*(\mathbb{C}) \cong \mathbb{Z}[0]$ and $\K_*(\mathbb{C}[1])=\K_*(\mathrm{C}_0(\mathbb{R})) \cong \mathbb{Z}[1]$, we have 

\begin{equation} \label{cconditions}
   \R_{[a,b]}([c,d])\cong 
\begin{cases}
    \mathbb{C}& \text{if } c \leq a \leq d \leq b,\\
    \mathrm{C}_0(\mathbb{R})& \text{if } a< c,\, b<d \text{ and } c-1 \leq b,\\            
    0  & \text{otherwise.}
\end{cases}
\end{equation}
We will also frequently use the localized version of these representative elements. Recall that for $p \in \Spec\mathbb{Z}$, we put $\mathbb{F}_p=\Z/p$ for $p\neq 0$, and $\mathbb{F}_p=\mathbb Q$ for $p=0$.
\begin{definition}
For $p \in \Spec\mathbb{Z}$, let $\R_Y^p:=\R_Y\otimes \kappa(p)$, where $Y \in \LCX$ and $\kappa(p)$ is the unique $\mathrm{C}^*$-algebra in $\mathcal{B}$ with $\K_*(\kappa(p))\cong \mathbb{F}_p[0].$
\end{definition}

\begin{remark} \label{local-gener}
Since $\K_*(\R_Y)$ is torsion-free, the K\"{u}nneth formula gives $$\K_*(\R_Y^p) = \K_*(\R_Y\otimes \kappa(p)) \cong \K_*(\R_Y)\otimes_{\mathbb{Z}}\K_*(\kappa(p)).$$ So we get the same conditions as (\ref{conditions}) and (\ref{cconditions}) for $\R_{[a,b]}^p([c,d])$, but with $\mathbb{Z}[i]$ replaced by $\mathbb{F}_p[i]$ and $\mathbb{C}[i]$ by $\kappa(p)[i]$ for $i =0,1$.
\end{remark}

\begin{remark}
The conditions (\ref{conditions}) give a handy way to diagrammatically depict the filtrated $\K$-theory functor (see Figure \ref{fig:1}). This diagram represents the category $\mathcal{NT}$; it is understood to be filled with zeros outside the infinite ``strip". Every arrow $Y \rightarrow Z$ for $Y,Z \in \LCX$ represents the generator of the free group $\mathcal{NT}_*(Y,Z)$. Dashed arrows represent degree one maps. Then, Figure \ref{fig:1} illustrates how $\mathcal{NT}$ is represented by what we will call an \emph{invariant triangle diagram}, which maps to the flipped version of itself via degree shifting maps. Also, since $\K_*\big(\R_Y(Z)\big) \cong \mathcal{NT}_*(Y,Z)$, the $X$-$\mathrm{C}^*$-algebra $\R_Y$ is represented in this diagram by a \emph{``maximal box"} starting at $Y$, a subdiagram of all $Z \in \LCX$ to which the morphisms from $Y$ do not factor through zero. In other words, all $Z=[c,d]$ that satisfy the conditions (\ref{conditions}) are inside the maximal box starting at $Y=[a,b]$; we denote the set of all such~$Z$ by $B_Y$; so $B_Y:=\{Z \in \LCX \mid \K_*(\R_Y(Z)) \ncong 0 \}.$ By Remark \ref{local-gener}, the situation is analogous for $\R_Y^p, \, p \in \mathrm{Spec}\,\mathbb{Z}$.
\end{remark}

\begin{figure}
\begin{tikzpicture}[baseline= (a).base]
\node[scale=0.75]  (a) at (-0.1,0){

{\scalefont{0.7}
\begin{tikzcd}[column sep={0.88cm,between origins},row sep={0.88cm,between origins}]
\draw[overlay,densely dotted, red] (1.5,-5.7) -- (7.9,0.69) -- (14.4, -5.7) -- cycle;
\draw[overlay,densely dotted, red] (12, 0.6) -- (8.6,0.6) -- (15, -5.75) -- (16.3,-4.7);
\draw[overlay,densely dotted, red] (-0.3,-4.6)-- (0.9,-5.75) -- (7.3, 0.6) -- (3.9,0.6);
\draw[overlay,densely dashed, rounded corners,blue] (5.3,-3.45) -- (9.75,0.85) -- (12.25,-1.65) -- (7.9, -5.95)--cycle;

&&&&&... \arrow[dr]&&{[1,1]}\arrow[dashed,dr]&& {[1,n]} \arrow[dr]  && {[n,n]} \arrow[dr] && ... && \\
&&&&... \arrow[dr]&&{[1,2]}\arrow[ur] \arrow[dashed,dr]&& {[2,n]}\arrow[ur]\arrow[dr] &&{[1,n{-}1]}\arrow[dashed,ur] \arrow[dr] && {[n{-}1,n]}\arrow[ur]\arrow[dr] && ... &&\\
&&&... \arrow[dr]&&...\arrow[ur]\arrow[dashed,dr]&& {[3,n]} \arrow[ur]\arrow[dr]&&{[2,n{-}1]}\arrow[ur]\arrow[dr] && {[1,n{-}2]}\arrow[dashed,ur] \arrow[dr] && {[n{-}2,n]}\arrow[dr]\arrow[ur] && ... \\
&&... \arrow[dr]&& {[1,n{-}3]}\arrow[ur] \arrow[dashed,dr]&& ... \arrow[ur]\arrow[dr] &&...\arrow[ur]\arrow[dr] && ... \arrow[ur]\arrow[dr]&&\arrow[dr] ...\arrow[dashed,ur]  && ...\arrow[dr]\arrow[ur]&& ...\\
&... \arrow[dr] && {[1,n{-}2]}\arrow[ur]\arrow[dashed,dr] && {[n{-}2,n]} \arrow[ur]\arrow[dr]&&{[n{-}3,n{-}1]} \arrow[ur]\arrow[dr]&& {[n{-}4,n{-}2]}\arrow[ur]\arrow[dr] && ...\arrow[ur]\arrow[dr] &&{[1,3]}\arrow[dr] \arrow[dashed,ur] &&{[3,n]}\arrow[ur] \arrow[dr]&&...\\
... \arrow[dr]&&{[1,n{-}1]}\arrow[ur]\arrow[dashed,dr]&& {[n{-}1,n]}\arrow[ur] \arrow[dr] &&{[n{-}2,n{-}1]}\arrow[ur]\arrow[dr] && {[n{-}3,n{-}2]}\arrow[ur] \arrow[dr]&&...\arrow[ur]\arrow[dr] &&{[2,3]}\arrow[dr] \arrow[ur]&& {[1,2]} \arrow[dashed,ur] \arrow[dr]&&{[2,n]}\arrow[ur]\arrow[dr]&& ...\\
&{[1,n]}\arrow[ur]&& {[n,n]} \arrow[ur] &&{[n{-}1,n{-}1]}\arrow[ur] && {[n{-}2,n{-}2]}\arrow[ur] && ...\arrow[ur] && {[3,3]}\arrow[ur]&&{[2,2]} \arrow[ur] &&{[1,1]} \arrow[dashed,ur] &&{[1,n]}\arrow[ur] &&
\end{tikzcd}
}
};
\end{tikzpicture}
\caption{The invariant triangle is marked with dotted lines. The dashed square represents the box $B_Y$ for $Y=[n-3,n-1]$.} \label{fig:1}

\end{figure}
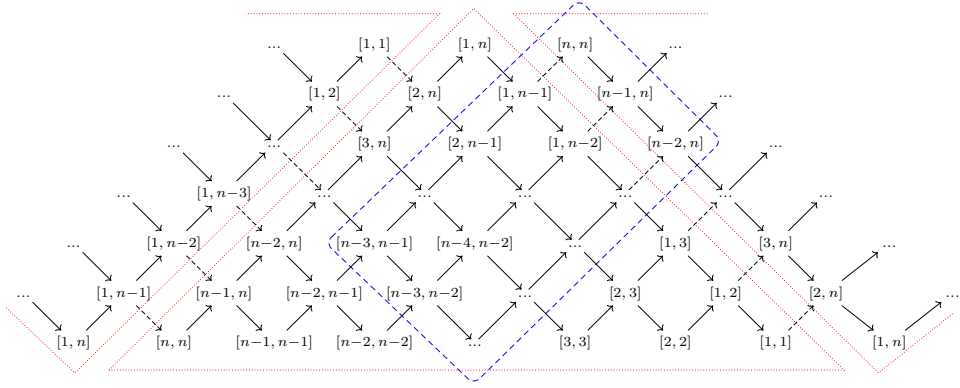

For classification purposes, we will use the characterization of $\mathcal{NT}$-modules in the image of filtrated $\mathrm{K}$-theory functor. These modules should have exactness properties coming from 6-term exact sequences of $\K$-theory. So we call an $\mathcal{NT}$-module $M$ exact if the chain complexes $$\cdots \longrightarrow M(U)\longrightarrow M(Y) \longrightarrow M(Y\setminus U) \longrightarrow M(U) \longrightarrow \cdots $$ are \emph{exact} for all $Y \in \LCX,\, U \in \mathbb{O}(Y)$ with maps coming from the generators in \eqref{conditions}.

Meyer-Nest show that exact modules also behave nicely homologically:

\begin{theorem}[Meyer-Nest \cite{MN12}] \label{proj-res}
Let $M \in \mathfrak{Mod}(\mathcal{NT})_{\mathrm{c}}$. Then $M=\mathrm{FK}(A)$ for some $A \in \mathfrak{KK}(X)$ if and only if $M$ is exact and if and only if $M$ has a projective resolution of length $1$ in $\mathfrak{Mod}(\mathcal{NT})_{\mathrm{c}}$.
\end{theorem}

The \emph{free $\mathcal{NT}$-module on $Y$}, for $Y \in \LCX$, is defined by $$Q_Y(Z):=\mathcal{NT}_*(Y,Z) \qquad \text{for every }Z\in \LCX.$$ An $\mathcal{NT}$-module is \emph{free} if it is isomorphic to a direct sum of degree-shifted free modules $Q_Y[j],\,j\in \Z/2$.

\begin{theorem}[Meyer-Nest {\cite[Theorem 3.12]{MN12}}] \label{splitt}
Let $M \in \mathfrak{Mod}(\mathcal{NT})_{\mathrm{c}}$. Then $M$ is a free $\mathcal{NT}$-module if and only if $M(Y)$ is a free abelian group for all $Y \in \LCX$ and $M$ is exact.
\end{theorem}

This theorem is a consequence of the fact that in case $M(Y)$ is free for all $Y \in \LCX$, a 1-step projective resolution of $M$ degenerates to a length-zero resolution, making $M$ itself projective and as a consequence free. 

For our classification, we will use the localized version of Theorem \ref{splitt}. For $p \in \Spec \Z$, let $$\mathcal{NT}^p:=\mathcal{NT}\otimes_\Z \mathbb{F}_p[0].$$
\begin{definition}
The \emph{free $\mathcal{NT}^p$-module on $Y$}, for $Y \in \LCX$ and $p \in \Spec \Z$, is defined by $$Q_Y^p(Z):=\mathcal{NT}^p_*(Y,Z)=\mathcal{NT}_*(Y,Z) \otimes_\Z \mathbb F_p[0] \qquad \text{for every }Z\in \LCX.$$ An $\mathcal{NT}^p$-module is \emph{free} if it is isomorphic to a direct sum of degree-shifted free modules $Q_Y^p[j],\,j\in \Z/2$.
\end{definition}

Even though the following theorem is not proved by Meyer-Nest, we give it here without a proof; the reason is that the proof is word by word the same as for Theorem \ref{splitt}, one just needs to replace the ring $\mathcal{NT}$ with $\mathcal{NT}^p$.

\begin{theorem} \label{reward}
Let $M \in \mathfrak{Mod}(\mathcal{NT})_{\mathrm{c}}$ and $p \in \Spec \Z$. Then $M$ is a free $\mathcal{NT}^p$-module if and only if $M(Y)$ is an $\mathbb{F}_p$-vector space for all $Y \in \LCX$ and $M$ is exact.
\end{theorem}

Recall that a \emph{multiset} is a collection of objects in which the elements are allowed to repeat.

As an easy corollary of Theorem \ref{reward} we get:

\begin{cor} \label{split}
Let $M \in \mathfrak{Mod}(\mathcal{NT})_{\mathrm{c}}, \;p \in \Spec \mathbb{Z}$ and $j=0,1$. Then $M\cong \bigoplus_{Y\in J}\mathrm{FK}\big(\R_Y^p\big)[j_Y]$ for some \textup{(}possibly countably infinite\textup{)} multiset $J$ with elements from $\LCX$ if and only if $M(Y)$ is an $\mathbb{F}_p$-vector space for all $Y \in \LCX$ and $M\cong \mathrm{FK}(A)$ for some $A \in \mathfrak{KK}(X)$.
\end{cor}

\begin{proof}
Let $M(Y)$ be an $\mathbb{F}_p$-vector space for all $Y \in \LCX$ and $M\cong \mathrm{FK}(A)$ for some $A \in \mathfrak{KK}(X)$. By~Theorem \ref{proj-res} the latter condition means that $M$ is exact. Then by Theorem \ref{reward}, conditions (\ref{conditions}) and Remark \ref{local-gener}
\begin{align*} M(Z) & \cong \bigoplus_{Y \in J}\mathcal{NT}^p_*(Y,Z)[j_Y] \cong \bigoplus_{Y \in J}\mathcal{NT}_*(Y,Z)[j_Y]\otimes_\Z \mathbb F_p[0] \\ & \cong \bigoplus_{Y \in J} \K_*\big(\R_Y(Z)\big)[j_Y]\otimes_\Z \mathbb F_p[0]   \cong \bigoplus_{Y \in J} \K_*\big(\R_Y^p(Z)\big)[j_Y].
\end{align*} 

Since $M=\bigoplus_{Z\in \LCX}M(Z)$, the definition of filtrated $\K$-theory gives
\begin{align*} 
M & \cong \bigoplus_{Z\in \LCX}\bigoplus_{Y \in J} \K_*\big(\R_Y^p(Z)\big)[j_Y] \cong  \bigoplus_{Z\in \LCX}\bigoplus_{Y \in J}\mathrm{FK}_Z(\R_Y^p)[j_Y] \\ 
  & \cong \bigoplus_{Y \in J}\bigoplus_{Z\in \LCX}\mathrm{FK}_Z(\R_Y^p)[j_Y] \cong  \bigoplus_{Y \in J} \mathrm{FK}(\R_Y^p)[j_Y].
\end{align*}

The reverse implication follows directly from Remark \ref{local-gener}.
\end{proof}

\section{Properties of $\mathcal{B}(X)$ and cohomological support} \label{Properties of boot}
In this section, we collect some facts that will be used later to prove the classification theorem.

\subsection{Some general results for $\mathcal{B}(X)$}

In this subsection, $X$ denotes an arbitrary topological space.

\begin{definition}
For an abelian group $G$, let $\kappa(G)$ be the unique object in $\mathcal{B}:=\mathcal{B}(\{*\})$ with $\K_0(\kappa(G))=G$ and $\K_1(\kappa(G))=0$.
\end{definition}
For example, in this notation $\kappa(p)=\kappa(\mathbb{F}_p).$
\begin{lemma}
$\kappa(-)$ has the following properties: \\
(i) $\kappa(\bigoplus_{i \in I}G_i ) \cong \bigoplus_{i \in I}\kappa(G_i)$; \\
(ii) Let $(G_i, f_j^i)$ be a countable inductive system and $(\kappa(G_i), \alpha_j^i)$ its lift by $\K$\nobreakdash-theory. Then $ \kappa(\varinjlim G_i) \cong \holim \kappa(G_i).$
\end{lemma}

\begin{proof}
(i) follows from additivity of K-theory. \\
(ii) By definition, the homotopy limit fits in an exact triangle $$ \Sigma \holim \kappa(G_i)\longrightarrow \bigoplus \kappa(G_i) \xrightarrow{\text{id} - \text{shift}_\alpha} \bigoplus \kappa(G_i)  \longrightarrow \holim \kappa(G_i).$$
After applying the K-theory functor and decomposing the resulting exact sequence into short exact sequences, we get $$\coker(\text{id} - \text{shift}_f) \hookrightarrow \K_*( \holim \kappa(G_i)) \twoheadrightarrow \ker(\text{id} - \text{shift}_f).$$ Now $\ker(\text{id} - \text{shift}_f[1]) \cong 0$ and $\coker(\text{id} - \text{shift}_f) \cong \varinjlim G_i$ by definition.
\end{proof}

\begin{lemma} \label{ideal}
 Let $\mathcal{S} \subseteq \mathcal{B}(X)$ be a localizing subcategory. For any $A \in \mathcal{S}$ and $G$ a countable abelian group, we have $A \otimes \kappa(G) \in \mathcal{S}$.
\end{lemma}

\begin{proof}
First let $G$ be finitely generated. Then $G \cong \mathbb{Z}^n \oplus \mathbb{Z}/{p_1^{i_1}} \cdots
\oplus \mathbb{Z}/{p_m^{i_m}}$. Now $A \otimes \kappa(\mathbb{Z}) \cong A\otimes \mathbb{C} \cong A \in \mathcal{S}$. Next consider the short exact sequence of $\mathbb{Z}/2$-graded abelian groups
\begin{equation*}
0\rightarrow \mathbb{Z}[0]\xrightarrow{p_k^{i_k}}\mathbb{Z}[0] \rightarrow \mathbb{Z}/{p_k^{i_k}}[0] \rightarrow 0.
\end{equation*}  
It lifts to a unique triangle in $\mathcal{B}$, namely,
\begin{equation} \label{tri}
\Sigma \kappa(\mathbb{Z}/{p_k^{i_k}}) \longrightarrow \mathbb{C} \longrightarrow \mathbb{C} \longrightarrow \kappa(\mathbb{Z}/{p_k^{i_k}}).
\end{equation}
Tensoring 
\eqref{tri} with $A$ leads to the triangle $$ \Sigma \big( A \otimes \kappa(\mathbb{Z}/{p_k^{i_k}})\big) \longrightarrow A \longrightarrow A \longrightarrow A \otimes \kappa(\mathbb{Z}/{p_k^{i_k}}).$$ We conclude that $A \otimes \kappa(\mathbb{Z}/{p_k^{i_k}}) \in \mathcal{S}$. Thus $A \otimes \kappa(G) \in \mathcal{S}$.
Now let $G$ be an arbitrary countable abelian group. Then $G \cong \varinjlim_{j \in \mathbb{N}} H_j,$ where $H_j \subseteq G$ are finitely generated subgroups. We have 
\begin{equation*}
A \otimes\kappa(G) \cong A \otimes \kappa(\varinjlim H_j) \cong \holim A \otimes \kappa(H_j) \in \mathcal{S}. \qedhere 
\end{equation*}
\end{proof}

\begin{cor}
 For any $D \in  \mathcal{B}$ and $A \in  \mathcal{B}(X)$, if $A \in \mathcal{S}$ then $A \otimes D \in \mathcal{S}$.
\end{cor}
\begin{proof}
 $D \cong \kappa(G) \oplus \kappa(H)[1]$ for some abelian groups $G$ and $H$, namely, $G=\K_0(D)$ and $H=\K_1(D)$.
\end{proof} 

\begin{lemma}
 \label{gener}
 For all $A \in \mathcal{B}(X)$, $ \langle A \rangle \cong \langle A\otimes \kappa(p) \mid p \in \Spec \mathbb{Z} \rangle.$
\end{lemma}
\begin{proof}
By Lemma \ref{ideal}, $A \otimes \kappa(\mathbb{Q}), A \otimes \kappa(\mathbb{Q/Z}) \in \langle A\rangle$. Moreover, there is an exact triangle $$ \Sigma \big(A \otimes \kappa(\mathbb{Q/Z})\big)\longrightarrow A \longrightarrow
 A \otimes \kappa(\mathbb{Q}) \longrightarrow A \otimes \kappa(\mathbb{Q/Z}). $$ So $ \langle A \rangle \cong \langle A\otimes 
 \kappa(\mathbb{Q}),A\otimes \kappa(\mathbb{Q/Z}) \rangle $. 
 We also have the isomorphisms 
\begin{align*}
 A \otimes \kappa(\mathbb{Q/Z}) \cong A \otimes \bigoplus_{p \text{ prime}}\kappa\big(\mathbb{Z}\big[\frac{1}{p}\big]/\mathbb{Z}\big) & \cong \bigoplus_{p \text{ prime}} A\otimes \kappa\big(\varinjlim_{n}\mathbb{Z}/p^n\mathbb{Z}\big) \\ 
 & \cong \bigoplus_{p \text{ prime}} \holim A\otimes\kappa(\mathbb{Z}/p^n\mathbb{Z}).
 \end{align*}
These isomorphisms, together with the exact triangles $$\Sigma\Big(A\otimes\kappa\big(\mathbb{Z}\big[\frac{1}{p}\big]/\mathbb{Z}\big)\Big) \rightarrow A\otimes\kappa(\mathbb{Z}/p^n\mathbb{Z}) \rightarrow A\otimes\kappa\big(\mathbb{Z}\big[\frac{1}{p}\big]/\mathbb{Z}\big)
\xrightarrow{\text{id}\otimes \tilde{p}^n} A\otimes\kappa\big(\mathbb{Z}\big[\frac{1}{p}\big]/\mathbb{Z}\big) $$ and $$\Sigma \big(A\otimes \kappa(\mathbb{Z}/p^m\mathbb{Z})\big) \rightarrow A\otimes \kappa(\mathbb{Z}/p^n\mathbb{Z}) \xrightarrow{\text{id}\otimes \tilde{p}^m}
 A\otimes \kappa(\mathbb{Z}/p^{n+m}\mathbb{Z}) \rightarrow A\otimes \kappa(\mathbb{Z}/p^m\mathbb{Z})$$ imply that 
 \begin{equation*}
 \langle A \otimes \kappa(\mathbb{Q/Z}) \rangle \cong \langle A\otimes\kappa(p) \mid p \in \Spec {\mathbb{Z}\backslash\{0\}}\rangle. \qedhere 
 \end{equation*}
 \end{proof}

 \subsection{Cohomological support}
 Recall that every abelian group has a one-step minimal injective resolution, which is unique up to isomorphism. Also, every injective abelian group is isomorphic 
 to a direct sum of \emph{indecomposable} ones, namely, $\mathbb{Q}$ and $\mathbb{Z}[\frac{1}{p}]/\mathbb{Z}$, where $p$ is a prime number~\cite{Mat58}. 
 
 Let $p \in \Spec \mathbb{Z}$. We say that $p$ appears in a minimal injective resolution of the abelian group $G$ if $\mathbb{Z}[\frac{1}{p}]/\mathbb{Z}$ for $p \neq0$, and $\mathbb{Q}$ for $p=0$,
 appears in degree zero or one in the direct sum decomposition of the minimal injective resolution of $G$. We define $$\supp_\mathbb{Z} G :=\{ p \in \Spec \mathbb{Z} \mid \text{$p$ appears in a minimal injective resolution of $G$} \}.$$ 

 \begin{lemma}
\label{iff}
Let $A\in \mathcal{B}$. Then $A \otimes \kappa(p) \ncong 0$ if and only if $p \in \supp_\mathbb{Z} \K_*(A).$
\end{lemma}

\begin{proof}
First, assume $p \in \supp_\mathbb{Z} \K_*(A) $ and $ p \neq 0$; that is, $\mathbb{Z}[\frac{1}{p}]/\mathbb{Z}$ appears in some degree as a direct summand of $M_0$ or
$M_1$,  where $\K_*(A) \hookrightarrow M_0 \twoheadrightarrow M_1$ is a minimal injective resolution of $\mathbb{Z}/2$-graded abelian groups. If it appears in
$M_0$ in degree $k$, then $\im (\K_*(A))\cap \Sigma^k(\mathbb{Z}[\frac{1}{p}]/\mathbb{Z}) \ncong \{0\}$ (here we use $\Sigma$ to denote the shift functor, in order not to confuse it with adjoining an element) because $M_0$ is an essential extension of $\K_*(A)$. So $\K_*(A)$ contains an isomorphic copy of
 $\Sigma^k(\mathbb{Z}[\frac{1}{p}]/\mathbb{Z})$ or $\Sigma^k(\mathbb{Z}/p^n)$ for some $n \in \mathbb{N}$ (arbitrary subgroup of  $\Sigma^k(\mathbb{Z}[\frac{1}{p}]/\mathbb{Z})$). Thus $\K_*(A)\xrightarrow{p} \K_*(A)$ is not an
isomorphism. Now if $\Sigma^k(\mathbb{Z}[\frac{1}{p}]/\mathbb{Z})$ appears as a direct summand in $M_1$, but not in $M_0$, then $M_0 \xrightarrow{p} M_0$ is an isomorphism. If we assume that $\K_*(A)\xrightarrow{p} \K_*(A)$
is also an isomorphism, then, by the Five Lemma, so is $M_1 \xrightarrow{p} M_1$, which is a contradiction. So, if $p \neq 0$ and $p \in \supp_\mathbb{Z} \K_*(A) $ then  $\K_*(A)\xrightarrow{p} \K_*(A)$ is not 
an isomorphism. Therefore, the lift of this map $A\xrightarrow{\tilde{p}}A$ is also not an isomorphism. So, $\text{cone}(A\xrightarrow{\tilde{p}}A) \cong A\otimes \kappa(p) \ncong 0$.

Conversely, if $A\otimes \kappa(p)\cong \text{cone}(A\xrightarrow{\tilde{p}}A) \ncong 0$, then $\K_*(A)\xrightarrow{p} \K_*(A)$ is not an isomorphism. By the Five Lemma, one of $M_0 \xrightarrow{p} M_0$ or
$M_1 \xrightarrow{p} M_1$ is not an isomorphism as well; so $\mathbb{Z}[\frac{1}{p}]/\mathbb{Z}$ has to appear in some degree in $M_0$ or $M_1$.

Now consider $p=0$. Then $\mathbb{Q}$ appears in some degree $k$ as a direct summand of $M_0$ or $M_1$. In the first case,  $\im(\K_*(A))\cap \Sigma^k(\mathbb{Q}) \ncong \{0\}$, meaning that $\K_*(A)$ contains a torsion-free
subgroup, thus $\K_*(A)\otimes \mathbb{Q} \ncong 0$. However, we also know that $\K_*(A\otimes \kappa(0)) \cong \K_*(A)\otimes \mathbb{Q}$ by the K{\"u}nneth formula. If $\mathbb{Q}$ does not appear in $M_0$, then $M_0 \otimes \mathbb{Q} \cong 0$ and tensoring the minimal injective resolution with $\mathbb{Q}$ and using flatness
of $\mathbb{Q}$, we conclude that also $M_1 \otimes \mathbb{Q} \cong 0$ and thus  $\mathbb{Q}$ does not appear as a direct summand in $M_1$ either.

Conversely, if $A\otimes \kappa(0)\ncong 0$, then $ \K_*(A\otimes \kappa(0)) \cong \K_*(A) \otimes \mathbb{Q} \ncong 0$. As above, tensoring the minimal injective resolution of $\K_*(A)$ with $\mathbb{Q}$, gives $M_0 \otimes \mathbb{Q} \ncong 0$, so $0 \in \supp_\mathbb{Z} \K_*(A)$.
\end{proof}

\section{Localizing subcategories in the totally ordered case} \label{Localizing subcategories}

In this section, we restrict our attention to finite spaces with totally ordered lattice of open subsets. As observed in Section~\ref{prelim}, this amounts to considering $X=\{1,\dots,n\}$ totally ordered by $\leq$, where a subset is open if and only if it is of the form $ [a,n]:=\{x \in X \mid a\leq x\leq n\}$, $a\in X$. Then locally closed subsets are those of the form $[a,b]$ with $a\leq b$ and $a,b \in X$. The set of non-empty locally closed subsets is denoted by $\LCX$.

\begin{definition} 
Let $\mathcal{L} \subseteq  \mathcal{B}(X)$ be a localizing subcategory and $Y \in \LCX$. Define $\mathrm{U}_Y^{\mathcal{L}} \subseteq  \Spec \mathbb{Z}$ by
$$\mathrm{U}_Y^{\mathcal{L}}:=\{p \in  \Spec \mathbb{Z} \mid p \in \supp_\mathbb{Z}\K_*(A(Y)) \text{  for some  } A\in \mathcal{L} \}.$$  
\end{definition}

\begin{remark} \label{support}
In the introduction, we defined the support of an object $A \in \mathcal L$ in a localizing subcategory $\mathcal L \subseteq \mathcal B(X)$ as 
$$\supp A=\{ (p,Y) \mid \mathrm{K}_*(A(Y);\mathbb{F}_p) \neq 0 \},$$ and the support of $\mathcal L$ as $\supp \mathcal{L}=\bigcup_{A \in \mathcal{L}}\supp A.$

If $A \in \mathcal{B}$, we may set $\mathrm{K}_*(A;\mathbb{F}_p):=\K_*(A\otimes \kappa(p))$. Since the classical K{\"u}nneth sequence for $\K$-theory splits, $\K_*(A\otimes \kappa(p))$ is an $\mathbb{F}_p$-vector space. 

Thus, by Lemma \ref{iff}, for a localizing subcategory $\mathcal L \subseteq \mathcal B(X)$, $$\mathrm{U}_Y^{\mathcal{L}}=\{p \in  \Spec \mathbb{Z} \mid (p,Y) \in \supp \mathcal{L} \}.$$
\end{remark} 

We will prove that these sets are not independent: for any $Y \in \LCX$ and $\mathcal{L}\subseteq \mathcal{B}(X)$ a localizing subcategory, if $p \in \mathrm{U}_Y^{\mathcal{L}}$ then there exists a maximal box $$B_Z=\{W \in \LCX \mid \K_*(\R_Z(W)) \ncong 0 \}=\{W \in \LCX \mid \K_*(\R_Z^p(W)) \ncong 0 \},$$ such that $Y \in B_Z$ and $p \in \mathrm{U}_V^{\mathcal{L}}$ for all $V \in B_Z.$ In other words, we have
 
 \begin{lemma} \label{u}
  For every localizing subcategory $\mathcal{L} \subseteq \mathcal{B}(X)$ and $Y \in \LCX$, $$\mathrm{U}_Y^{\mathcal{L}} =
  \bigcup_{\substack{Z \in \LCX: \\ Y \in B_Z}} \bigcap_{ 
   V \in B_Z} \mathrm{U}_V^\mathcal{L}.$$
  \end{lemma}
 
 \begin{proof}
  
  First assume $p \in \bigcup_ {Z: Y \in B_Z} \bigcap_{ V \in B_Z} \mathrm{U}_V^\mathcal{L}.$ Then $p \in \bigcap_{V \in B_Z} \mathrm{U}_V^\mathcal{L}$ for some~$Z$ with $Y \in B_Z$. But then $\mathrm{U}_Y^{\mathcal{L}}$ is itself in this intersection. Thus $p \in \mathrm{U}_Y^{\mathcal{L}}$.
  
  Now take $p \in \mathrm{U}_Y^\mathcal{L}$. By definition, there is $A \in \mathcal{L}$ with $p \in \supp_\mathbb{Z}\K_*(A(Y))$. Lemma~\ref{iff} implies that $\text{cone}(A(Y)\xrightarrow{\tilde{p}}A(Y)) \cong A(Y) \otimes \kappa(p) \ncong 0$. This implies
  that $\text{cone}(A\xrightarrow{\tilde{p}}A) \ncong 0$ because $\text{cone}(A\xrightarrow{\tilde{p}}A)(Y) \cong \text{cone}(A(Y)\xrightarrow{\tilde{p}}A(Y)).$ 
  However, $\text{FK}(\text{cone}(A\xrightarrow{\tilde{p}}A))(Z) \cong \K_*(\text{cone}(A\xrightarrow{\tilde{p}}A)(Z))
  \cong \K_*(A(Z)\otimes \kappa(p))$ is an $\mathbb{F}_p$-vector space for any $Z \in \LCX$ and $p\in \Spec \mathbb{Z}$
  because the classical K{\"u}nneth sequence for $\K$-theory splits. Thus, by Corollary \ref{split}, there exists a multiset $I\subseteq \LCX$ such that $\text{FK}(\text{cone}(A\xrightarrow{\tilde{p}}A)) \cong \bigoplus_{Z \in I} \text{FK}(\R_Z^p) \cong  \text{FK}(\bigoplus_{Z \in I}\R_Z^p).$

Now we can use Corollary \ref{uct-cor} and Theorem \ref{proj-res} to lift the isomorphism of filtrated K-theories to an isomorphism in $\mathcal{B}(X)$. In other words, $\text{cone}(A\xrightarrow{\tilde{p}}A) \cong \bigoplus_{Z \in I}\R_Z^p$.
Since $\K_*\big(\text{cone}(A\xrightarrow{\tilde{p}}A)(Y)\big)\ncong 0$, there is $ Z \in I$ such that $\R_Z^p(Y) \ncong 0$. Since $\mathcal{L}$ is localizing, it contains all the direct summands of its objects. Thus $\R_Z^p \in \mathcal{L}$.

Since $\R_Z^p (V) = (\R_Z \otimes \kappa(p))(V) \cong \R_Z(V) \otimes \kappa(p)$, the following implications hold for any $V \in \LCX$: 
\begin{align*}
V \in B_Z & \iff \K_*(\R_Z^p (V))\ncong 0 \\
& \iff \K_*(\R_Z^p (V)) \text{  is isomorphic to  } \mathbb{F}_p[i],\, i \in \mathbb{Z}/2 \\
& \implies p\in \supp_\mathbb{Z}(\K_*(\R_Z^p (V)) \\
& \implies p \in \mathrm{U}_V^{\mathcal{L}}, 
\end{align*}  
where the first two equivalences hold because $\K_*(\R_Z (V)\otimes \kappa(p)) \cong \K_*(\R_Z (V))\otimes \K_*(\kappa(p))$  
by to the K{\"u}nneth formula and because $\K_*(\R_Z (V)) \cong \mathbb{Z}[i]$ for some $i \in \mathbb{Z}/2$.
In particular, these implications mean that $p \in \bigcap_{V \in B_Z} \mathrm{U}_V^\mathcal{L},$ and since $Y \in B_Z$,
we get 
\begin{equation*}
p \in \bigcup_ {Z: Y \in B_Z} \bigcap_{ 
   V \in B_Z} \mathrm{U}_V^\mathcal{L}. \qedhere
   \end{equation*}
  \end{proof}

\begin{remark} \label{doIneed}
Since $\mathrm{U}_Y^\mathcal{L}$ is itself in every intersection over which we are taking the unions in $\bigcup_ {Z, Y\in B_Z} \bigcap_{V \in B_Z} \mathrm{U}_V^\mathcal{L}$, we can factor it out and get $$\mathrm{U}_Y^{\mathcal{L}} = \mathrm{U}_Y^{\mathcal{L}} \cap
  \bigcup_{\substack{Z \in \LCX: \\ Y \in B_Z}} \bigcap_{\substack{ V \neq Y 
  \\ V \in B_Z}} \mathrm{U}_V^\mathcal{L}.$$ Therefore, Lemma \ref{u} is equivalent to
  
  $$\mathrm{U}_Y^{\mathcal{L}} \subseteq
  \bigcup_{\substack{Z \in \LCX: \\ Y \in B_Z}} \bigcap_{\substack{V \neq Y
  \\ V \in B_Z}} \mathrm{U}_V^\mathcal{L}.$$
\end{remark}

\begin{definition}
 For a localizing subcategory $\mathcal{L} \subseteq \mathcal{B}(X)$ and $Y \in\LCX$, define $\mathrm{V}_Y^\mathcal{L} \subseteq \Spec \mathbb{Z}$
 by $$\mathrm{V}_Y^\mathcal{L}:=\{ p \in \Spec \mathbb{Z} \mid \R_Y^p \in \mathcal{L} \}.$$
 
\end{definition}

\begin{lemma} \label{vgene}
For any localizing subcategory $\mathcal{L} \subseteq \mathcal{B}(X)$, $$\mathcal{L} \cong \langle \R_Y^p \mid p \in \mathrm{V}_Y^{\mathcal{L}},\, Y \in \LCX \rangle.$$
\end{lemma}

\begin{proof}
$\langle \R_Y^p \mid p \in \mathrm{V}_Y^{\mathcal{L}},\, Y \in \LCX \rangle \subseteq \mathcal{L}$ by definition.

Now if $A \in \mathcal{L}$, then $A \in  \langle A\otimes \kappa(p)\mid p \in \Spec \mathbb{Z} \rangle$ by Lemma \ref{gener}. But as was shown in the proof of Lemma \ref{u}, $A\otimes \kappa(p)\cong \text{cone}(A\xrightarrow{\tilde{p}}A)\cong \bigoplus_{Z \in I}\R_Z^p$ with $\R_Z^p \in \mathcal{L}$ for $Z \in I$. Thus also $A \in  \langle \R_Y^p \mid p \in \mathrm{V}_Y^{\mathcal{L}},\, Y \in \LCX \rangle$.
\end{proof}

By Lemma \ref{vgene}, specifying the sets $\mathrm{V}_Y^\mathcal{L} \subseteq \Spec \Z$ for all $Y \in \LCX$ completely determines the localizing subcategory $\mathcal{L}$. Our aim is to show that the sets $\mathrm{U}_Y^\mathcal{L}$ for all $Y \in \LCX$ determine the sets $\mathrm{V}_Y^\mathcal{L}$,
and thus  $\mathcal{L}$ itself. However, in order to show this, we  first need to prove some preliminary statements.

\begin{lemma}
 \label{gen}
 Let $Y,V,W \in \LCX$. If $Y$ equals $V \cap W$ or $V \cup W$ or $V \setminus W$, then $\R_Y \in \langle \R_V, \R_W \rangle.$
\end{lemma}

\begin{proof}

First, say $V \cup W \notin \LCX$. This implies $V \setminus W=V$ and $V \cap W = \emptyset$, trivially giving the assertion. The same way, if $V \setminus W \notin \LCX$, we must have $W \subset V$, thus $V \cup W = V$ and $V \cap W =W$, giving the result. Similarly, the assertion is trivial if $W \setminus V \notin \LCX$. So we assume $V \cup W,\;V \setminus W, \;W \setminus V  \in \LCX$. Write $V=[v_1,v_2]$ and $W=[w_1,w_2]$. Without loss of generality, we can also assume $v_1 \leq w_1,\; v_2 \leq w_2$ by exchanging $V$ and $W$ if necessary. However, since we sacrificed the symmetry, we have to prove the lemma for $Y=W \setminus V$ as well.

Let $Z \in \LCX$ and $U \in \mathbb{O}(Z)$. Since $\R$ is a $\mathrm{C}^*$-algebra over $X^\text{op}\times X$, there is a semi-split extension by *-homomorphisms $$\R_{Z \setminus U} \hookrightarrow \R_Z \twoheadrightarrow \R_U$$ which produces an exact triangle
$$\Sigma\R_U \rightarrow  \R_{Z \setminus U} \rightarrow \R_Z \rightarrow \R_U $$ in $\mathfrak{KK}(X)$.

Since $W\setminus V= [v_2+1,w_2]$ is open in $V \cup W=[v_1,w_2]$, $V\cap W=[w_1,v_2]$ is open in $V=[v_1,v_2]$ and $W=[w_1,w_2]$ is open in $V \cup W=[v_1,w_2]$, we get the following exact triangle 

$$ \Sigma \R_{W \setminus V} \rightarrow \R_{V} \rightarrow \R_{V \cup W} \rightarrow \R_{W \setminus V} $$ along with two exact triangles fitting in a commutative square

\begin{displaymath}
    \xymatrix{ \Sigma \R_{V \cap W} \ar[r] & \R_{V \setminus W} \ar[r] \ar@{=}[d] & \R_V \ar[d] \ar[r] & \R_{V \cap W}  &\\
               \Sigma \R_W \ar[r] & \R_{V \setminus W} \ar[r] & \R_{V \cup W} \ar[r] & \R_W & }
\end{displaymath}

By the octahedral axiom, there exists a map $\R_{V \cap W} \rightarrow \R_W$ such that the third square in the above diagram will be homotopy
cartesian; in other words, there is an exact triangle 
$$\Sigma \R_W \rightarrow \R_V \rightarrow \R_{V \cup W}\oplus \R_{V \cap W} \rightarrow \R_W $$ 
These four triangles show that $\R_{W \setminus V}, \R_{V \cap W},\R_{V \cup W},\R_{V \setminus W} \in \langle \R_V, \R_W \rangle$.
\end{proof}

Now we proceed to prove the key proposition.

\begin{proposition}
\label{uv}
 For a localizing subcategory $\mathcal{L} \subseteq \mathcal{B}(X)$, we have $$\mathrm{V}_Y^\mathcal{L}=\bigcap_{ 
  Z \in B_Y}\mathrm{U}_Z^\mathcal{L}.$$
\end{proposition}

\begin{proof}
If $p \in \mathrm{V}_Y^{\mathcal{L}}$, then $\R_Y^p \in \mathcal{L}$ by definition. Also, exactly as for Lemma  \ref{u},
\begin{align*}
Z \in B_Y & \iff \K_*(\R_Y^p (Z))\ncong 0 \\
& \iff \K_*(\R_Y^p (Z)) \text{  is isomorphic to  } \mathbb{F}_p[i]\text{  for  }  i \in \mathbb{Z}/2 \\
& \implies p\in \supp_\mathbb{Z}(\K_*(\R_Y^p (Z)) \\
& \implies p \in \mathrm{U}_Z^{\mathcal{L}}. 
\end{align*}  
Thus $p \in  \bigcap_{Z \in B_Y}\mathrm{U}_Z^\mathcal{L}.$

The opposite inclusion needs more work. Let $p \in \bigcap_{Z \in B_Y}\mathrm{U}_Z^\mathcal{L}.$ As in the proof of Lemma \ref{u}, this means that for any $Z \in \LCX$ with $Z \in B_Y$, there exists $W \in \LCX$ with $Z \in B_W$ and $\R_W^p \in \mathcal{L}$. Let $J \subseteq \LCX$ be the set of all such $W$'s. Tensoring with $\kappa(p)$ is an exact functor and commutes with coproducts. So
$\R_Y \in \langle \R_W \mid W\in J \rangle $ implies $\R_Y^p \in \langle \R_W^p \mid W\in J \rangle \subseteq \mathcal{L}$ and thus $p \in \mathrm{V}_Y^\mathcal{L}$. Therefore, it suffices to prove $\R_Y \in \langle \R_W \mid W\in J \rangle $.
 
 First, we show that $Y$ is covered by intervals in $J$. Let $Y=[a,b]$. For any $i \in [a,b]$, by (\ref{conditions}), we have
 $[1,i] \in B_{[a,b]}$ because $1 \leq a\leq i \leq b$. So we know that there exists $W \in J$ with $[1,i] \in B_{W}$. Let $W=[a_1,b_1]$. Since $[1,i] \in B_{[a_1,b_1]}$, again by (\ref{conditions}), there is only one possibility, namely
 $1 \leq a_1\leq i \leq b_1$, which means $i \in W$.

Now, let $M^i$ be the interval of minimal length such that $i \in M^i$ and $\R_{M^i} \in \langle \R_W \mid W\in J \rangle$. Such an interval is unique; if $N^i$ is another interval with the same properties, then $i \in M^i\cap N^i$, $\R_{M^i\cap N^i} \in \langle \R_W \mid W\in J \rangle$ by Lemma \ref{gen} and $|M^i\cap N^i| < |M^i|$, contradicting minimality.
 
We want to demonstrate that $M^i \subseteq Y$; because then $Y=\bigcup_{j \in Y}M^j$, and by Lemma \ref{gen}, $\R_Y \in \langle \R_W \mid W\in J \rangle$, concluding the proof of the proposition.
 
Let $M^i= [k,l]$. Assume $k < a$. Now, by (\ref{conditions}), $[k+1,i] \in B_{[a,b]}$ because $k+1 \leq a \leq i \leq b$. Therefore, there exists $W \in J$ with $[k+1,i] \in B_W$. Let $W= [c,d]$. Again by (\ref{conditions}), we have two possibilities:

\begin{description}[font=\normalfont \textit]
  \item[Case 1] $k+1\leq c \leq i \leq d$. Then $[c,d] \cap [k,l] =[c, \min\{d,l\}]$, and thus $\R_{[c, \min\{d,l\}]} \in \langle \R_W \mid W\in J \rangle$ by Lemma \ref{gen}. But $c\leq i \leq \min\{d,l\}$, thus $i \in [c,d] \cap [k,l]$. Moreover, $|[c,d] \cap [k,l]| < |[k,l]|$ because $k<c$ and $\min\{d,l\} \leq l$; this contradicts the minimality of $[k,l]$.
  
  \item[Case 2] $c<k+1,\,d<i,\,k \leq d$. Then $[k,l] \setminus [c,d] =[d+1, l]$ because $c \leq k,\, d<i \leq l$. Thus $\R_{[d+1, l]} \in \langle \R_W \mid W\in J \rangle$ by Lemma \ref{gen}. Since $d+1 \leq i \leq l$, $i \in [d+1,l]$. Moreover, $ |[d+1,l]| < |[k,l]|$ because $k<d+1$; this contradicts the minimality of $[k,l]$.
\end{description}
We conclude that $a \leq k$. Assume $b < l$. Now, by (\ref{conditions}), $[i+1,l] \in B_{[a,b]}$ because $a < i+1, \, b<l,\, i \leq b$. Therefore, there exists $W=[c,d] \in J$ with $[i+1,l] \in B_{[c,d]}$. Again, there are two cases to consider:
 
 \begin{description}[font=\normalfont \textit]
  \item[Case 1] $i+1\leq c \leq l \leq d$. Then $[k,l] \setminus [c,d] =[k, c-1]$ because $k < c,\, l\leq d $. Thus $\R_{[k, c-1]} \in \langle \R_W \mid W\in J \rangle$ by Lemma \ref{gen}. Since $k \leq i \leq c-1$, $i \in [k,c-1]$. Moreover, $ |[k,c-1]| < |[k,l]|$ because $c-1<l$; this contradicts the minimality of $[k,l]$.
  
  \item[Case 2] $c<i+1,\,d<l,\,i \leq d$. Then $[c,d] \cap [k,l] =[\max\{k,c\},d]$, and thus $\R_{[\max\{k,c\},d]} \in \langle \R_W \mid W\in J \rangle$ by Lemma \ref{gen}. But $\max\{k,c\} \leq i \leq d$, thus $i \in [c,d] \cap [k,l]$. moreover, $|[c,d] \cap [k,l]| < |[k,l]|$ because $d<l$; this contradicts the minimality of $[k,l]$.
\end{description}
 Finally, we have $a \leq k\leq l \leq b$; that is, $M^i \subseteq Y$. This finishes the proof of the proposition.
 \end{proof}
 
 Now we are ready to prove the main theorem of this section. We will restate it by concretely constructing the isomorphism. Let $m=|\LCX|$ be the number of non-empty intervals in $X$; that is, if $X$ has $n$ points, $m=\frac{n(n+1)}{2}$. 
 
\begin{theorem}
\label{maintt}
There is an inclusion-preserving isomorphism between localizing subcategories of $\mathcal{B}(X)$ and those elements $(\mathrm{U}_{Y_1}, \dots ,\mathrm{U}_{Y_m}) \in \mathcal{P}(\Spec \mathbb{Z})^m$ of the $m$-fold Cartesian product of subsets of the Zariski spectrum of the ring of integers, labeled by intervals  $Y_i \subseteq X$, which satisfy $\mathrm{U}_{Y_i} = \bigcup_{j, Y_i \in B_{Y_j}} \bigcap_{
 Y_k \in B_{Y_j}} \mathrm{U}_{Y_k}$ for all $i=1,\dots,m.$ The isomorphism and its inverse map are given by

\begin{align*}
 \mathcal{L}   & \longmapsto  \{ \mathrm{U}_{Y_i}^\mathcal{L} \}_{i=1}^m  \\
 \makebox[55pt][r]{$\displaystyle \langle \R_{Y_i}^p \mid p \in \bigcap_{Y_j \in B_{Y_i}} \mathrm{U}_{Y_j}, i=1,\dots,m  \rangle$} & \longmapsfrom \{ \mathrm{U}_{Y_i} \}_{i=1}^m.
 \end{align*}
 \end{theorem}

\begin{proof}
By Proposition \ref{uv}, the sets $\mathrm{U}_{Y_i}^\mathcal{L}$ determine the sets $\mathrm{V}_{Y_i}^\mathcal{L}$ and therefore, by Lemma \ref{vgene}, the localizing subcategory $\mathcal{L}.$ 

It remains to show that if $\mathcal{L}=\langle \R_{Y_i}^p \mid p \in \bigcap_{ 
 Y_j \in B_{Y_i}}\mathrm{U}_{Y_j} ,\; i=1,\dots,m    \rangle $, then $\mathrm{U_{Y_i}}=\mathrm{U_{Y_i}^\mathcal{L}}$ for all $i$.
 
 Let $p \in \mathrm{U}_{Y_i} $, then $p \in \bigcup_{j, Y_i \in B_{Y_j}} \bigcap_{ Y_k \in B_{Y_j}} \mathrm{U}_{Y_k}$. Therefore, there exists $j$ such that $\R_{Y_j}^p \in \mathcal{L}$ and $\R_{Y_j}^p(Y_i)\ncong 0$, since $Y_i \in B_{Y_j}$. This, in turn, implies that $p \in \mathrm{supp }_\mathbb{Z}\K_*(\R_{Y_j}^p(Y_i))$. Thus $ p \in \mathrm{U}_{Y_i}^\mathcal{L}$.
 
 Now let $p \in \mathrm{U}_{Y_i}^\mathcal{L} $. Then, as in the proof of Lemma \ref{u}, there exists $j$ such that 
$\R_{Y_j}^p \in \mathcal{L}$ and $\R_{Y_j}^p(Y_i)\ncong 0$. It follows that for any set of generators of $\mathcal{L}$, at least one generator has to not vanish at $Y_i$ because $\mathcal{L}$ contains an object not vanishing at $Y_i$ and exact triangles in $\mathcal{L}$ come from short semi-split exact sequences of $\mathrm{C}^*$-algebras over $X$. In particular, there must exist $k$ with $p \in \bigcap_{Y_l \in B_{Y_k}}\mathrm{U}_{Y_l}$ and $\R_{Y_k}^p \in \mathcal{L}$. Since $Y_i \in B_{Y_k}$, we get $p \in \mathrm{U}_{Y_i}$.
\end{proof}

\begin{remark}
Remark \ref{support} identifies the set $\{ \mathrm{U}_{Y_i}^\mathcal{L} \}_{i=1}^m $ with $\supp \mathcal{L}$. So Theorem \ref{maintt} shows that every localizing subcategory is uniquely determined by its support and describes which sets can appear as the support of a localizing subcategory.
\end{remark}

\subsection{Case of extensions}
To illustrate Theorem \ref{maintt}, let $X=\{1,2\}$ be the \emph{Sierpiński} space, a two-point topological space whose open sets are $$ \mathbb{O}(X)=\{ \emptyset, \{2\}, \{1,2\}\}.$$
The category of $\mathrm{C}^*$-algebras over $X$ is equivalent to the category of extensions of $\mathrm{C}^*$-algebras. We have three non-empty locally closed sets, $\LCX = \{\{1\},\{2\},\{1,2\}\}$. By Remark~\ref{doIneed}, the conditions on the sets $\mathrm{U}_Y$ for $Y \in \LCX$ translate to $$\mathrm{U}_{\{1\}} \subseteq \mathrm{U}_{\{2\}} \cup \mathrm{U}_{\{1,2\}} , \;\; \mathrm{U}_{\{2\}} \subseteq \mathrm{U}_{\{1\}} \cup \mathrm{U}_{\{1,2\}}, \;\; \mathrm{U}_{\{1,2\}} \subseteq \mathrm{U}_{\{1\}} \cup \mathrm{U}_{\{2\}}. $$ Therefore, we get:

\begin{cor} \label{extensions}
 There is a bijection between localizing subcategories of the Kasparov category of extensions of $\mathrm{C}^*$-algebras and those triples of
 subsets of $\Spec \mathbb{Z}$ which have the property that each one is inside the union of the other two. The bijection and its inverse map are given by
 \begin{gather*}
 \mathcal{L}  \longmapsto \begin{pmatrix}\{\supp_\mathbb{Z}\K_*(I)\mid I\triangleleft A \in \mathcal{L}\}\\ \{\supp_\mathbb{Z}\K_*(A)\mid I\triangleleft A \in \mathcal{L}\} \\ \{\supp_\mathbb{Z}\K_*(A/I)\mid I\triangleleft A \in \mathcal{L}\} \end{pmatrix}
\\
\left\langle
    \begin{array}{c|c}
      \kappa(p) \triangleleft \kappa(p) &  p \in  \mathrm{U}_{\{2\}} \cap \mathrm{U}_{\{1,2\}} \\
      0\triangleleft \kappa(q) &q \in \mathrm{U}_{\{1,2\}} \cap \mathrm{U}_{\{1\}} \\ 
      \kappa(s)[1] \triangleleft 0 & s \in \mathrm{U}_{\{1\}} \cap \mathrm{U}_{\{2\}} \\
    \end{array}
    \right\rangle   \longmapsfrom \begin{pmatrix}\mathrm{U}_{\{2\}}\\\mathrm{U}_{\{1,2\}}\\\mathrm{U}_{\{1\}} \end{pmatrix}. 
\end{gather*}
 \end{cor}

This example already demonstrates a difference between the classification of Theorem \ref{maintt} and other instances in the literature, where the triangulated category $\T$ in question carries an action of a commutative ring. In the latter case, as explained in the introduction, the lattice of localizing subcategories  $\mathrm{Loc}(\T)$ is isomorphic to the lattice of subsets of some topological space $Y$, where, in addition, the topology on $Y$ determines certain structure on $\mathrm{Loc}(\T)$. In this case, one can regard $Y$ as a good candidate for a topological space associated to $\T$. However, this construction is not possible for $\mathcal{B}(X)$.

\begin{theorem} \label{negat}
The lattice of localizing subcategories $\mathrm{Loc}\big(\mathcal{B}(X)\big)$ of the bootstrap category $\mathcal{B}(X)$ is not isomorphic to a sublattice of the subset lattice $\mathcal{P}(S)$ for any set $S$.
\end{theorem}

\begin{proof}

For a set $S$, any sublattice $L\subseteq \mathcal{P}(S)$ of a subset lattice is distributive; that is, $A \cap (B \cup C) = (A \cap B)\cup(A \cap C)$ for all $A,B,C \in L$.

By Corollary \ref{extensions}, elements of $\mathrm{Loc}\big(\mathcal{B}(X)\big)$ are characterized by triples of subsets of $\Spec \mathbb{Z}$, which have the property that each one is inside the union of the other two. In particular, for some prime $p \in \Spec \mathbb{Z}$, we have three localizing subcategories described by triples $(\{p\},\{p\},\emptyset)$, $(\emptyset,\{p\},\{p\})$ and $(\{p\},\emptyset,\{p\})$. However,
$(\{p\},\{p\},\emptyset) \wedge \big( (\emptyset,\{p\},\{p\}) \vee (\{p\},\emptyset,\{p\}) \big) = (\{p\},\{p\},\emptyset) \neq (\emptyset, \emptyset, \emptyset)= \big( (\{p\},\{p\},\emptyset) \wedge (\emptyset,\{p\},\{p\}) \big) \vee \big( (\{p\},\{p\},\emptyset) \wedge (\{p\},\emptyset,\{p\}) \big)$. Thus $\mathrm{Loc}\big(\mathcal{B}(X)\big)$ is not distributive.
\end{proof}

\section{Classification by noncrossing partitions} \label{Noncorssing partitions}

In this section, we describe the lattice of localizing subcategories of $\mathcal{B}(X)$ in another way, namely, by \emph{noncrossing partitions.}

\begin{definition}
For $p \in \Spec \mathbb{Z}$, we say that the localizing subcategory $\mathcal{L}$ is $p$-local if, for all $Y \in \LCX$, the set $\mathrm{U}_Y^\mathcal{L}$ is equal to $\{p\}$ or is empty.
\end{definition}

\begin{remark} \label{p-loc}
The $p\,$-local localizing subcategories are exactly the ones generated by $\R_{Y}^p$ for $Y \in I \subseteq \LCX$. Every localizing subcategory $\mathcal{L} \subseteq \mathcal{B}(X)$ can be uniquely represented by the $p\,$-local subcategories it contains, if we require that there is at most one (the largest) $p\,$-local subcategory for each $p \in \Spec \mathbb{Z}$ in this representation. This follows from Theorem \ref{maintt}, since the corresponding property is trivial for the sets $(\mathrm{U}_{Y_1}^{\mathcal{L}}, ... ,\mathrm{U}_{Y_m}^{\mathcal{L}}) \in \mathcal{P}(\Spec \mathbb{Z})^m$, where $m= |\LCX|$ is the number of non-empty intervals.
\end{remark}

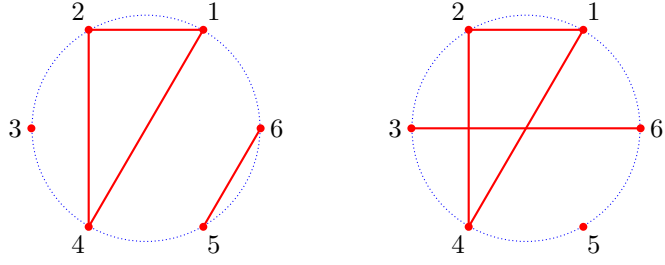
\begin{figure}
\tikzset{dot/.style={circle,fill=red,minimum size=3pt, inner sep=0pt}}
\begin{tikzpicture}

\draw[blue, densely dotted] (0,0) circle(1.5cm);
\node[regular polygon, regular polygon sides=6, minimum size=3cm] at (0,0) (A) {};

\foreach \i in {1,...,6}
    \node[ dot, fill=red] at (A.corner \i) {};
   
\draw[thick,red] (A.corner 1)--(A.corner 2)--(A.corner 4)--cycle;
\draw[thick,red] (A.corner 6)--(A.corner 5);
\foreach \i in {1,...,6}
    \node[anchor=(360/6)*\i+180] at (A.corner \i) {\i};

\draw[blue, densely dotted] (5,0) circle(1.5cm);
\node[regular polygon, regular polygon sides=6, minimum size=3cm] at (5,0) (B) {};

\foreach \i in {1,...,6}
    \node[ dot, fill=red] at (B.corner \i) {};
   
\draw[thick,red] (B.corner 1)--(B.corner 2)--(B.corner 4)--cycle;
\draw[thick,red] (B.corner 6)--(B.corner 3);
\foreach \i in {1,...,6}
    \node[anchor=(360/6)*\i+180] at (B.corner \i) {\i};

\end{tikzpicture}
\caption{The first picture shows the noncrossing partition $\{ \{1,2,4\},\{3\},\{5,6\}\}$ of the regular hexagon represented as vertices on a circle. The partition $\{ \{1,2,4\},\{3,6\},\{5\}\}$ on the second picture is crossing.} \label{fig:3}
\end{figure}

\subsubsection{Classical noncrossing partitions} 

A \emph{partition} of a given set of  $n$ elements is a collection of pairwise disjoint, nonempty subsets called \emph{blocks}, whose union is the entire set. Since being in the same block is an equivalence relation, we denote it by $\sim$. A partition of $\{1,\dots,n\}$ is \emph{noncrossing} if, when four elements with $1\leq a<b<c<d\leq n $
are such that $a\sim c$ and $b \sim d$, then the two blocks coincide, meaning $a\sim b \sim c \sim d $. The terminology comes from the fact that a noncrossing partition admits a planar representation as a partition of the vertices of a regular $n$-gon (labeled by $\{1,\dots,n\}$) with the property that the convex hulls of its blocks are pairwise non-crossing (see Figure \ref{fig:3}). The collection of noncrossing partitions of an $n$\nobreakdash-element set is denoted by $\mathrm{NC}_n$.

$\mathrm{NC}_n$ becomes a partially ordered set when partitions are ordered by \emph{refinement}: given partitions $\sigma,\tau \in \mathrm{NC}_n$, we say that $ \tau \leq \sigma$ if each block of $\sigma$ is contained in a block of $\tau$. For each $n$, the partially ordered set $\mathrm{NC}_n$ is a self-dual, bounded lattice with $C_n$ elements, where $C_n=\frac{1}{n+1}{{2n}\choose{n}}$ is the $n$th \emph{Catalan number}. Figure~\ref{fig:4} depicts this lattice for $n=4$. For the exposition of the classical theory of noncrossing partitions and the proof of these facts see~\cite[Chapter~4]{Arm09}.

\subsection{Classification}
Again let $X=\{1,2,\dots,n\}$ with the Alexandrov topology.

\begin{figure}
  \tikzset{dot/.style={circle,fill=red,minimum size=2pt, inner sep=0pt}}
\begin{tikzpicture}
\pgfmathsetmacro\r{0.6} 
\pgfmathsetmacro\h{2} 
\pgfmathsetmacro\d{0.3} 

\foreach \j in {0,...,5}
{
\draw[blue, densely dotted] (\j*2*\r+ \j*\d,0) circle(\r cm);
\node[regular polygon, regular polygon sides=4, minimum size=2*\r cm] at (\j*2*\r+ \j*\d,0) (\j) {};
\foreach \i in {1,...,4}
    \node[dot, fill=red] at (\j.corner \i) {};
}

\foreach \j in {6,...,11}
{
\draw[blue, densely dotted] (\j*2*\r+ \j*\d-6*2*\r-6*\d,\h) circle(\r cm);
\node[regular polygon, regular polygon sides=4, minimum size=2*\r cm] at (\j*2*\r+ \j*\d-6*2*\r-6*\d,\h) (\j) {};
\foreach \i in {1,...,4}
    \node[dot, fill=red] at (\j.corner \i) {};
}

\draw[red] (0.corner 1)--(0.corner 4)  (0.corner 2)--(0.corner 3);
\draw[red] (1.corner 1)--(1.corner 3)--(1.corner 4)--cycle;
\draw[red] (2.corner 1)--(2.corner 2)--(2.corner 3)--cycle;
\draw[red] (3.corner 1)--(3.corner 2)--(3.corner 4)--cycle;
\draw[red] (4.corner 2)--(4.corner 3)--(4.corner 4)--cycle;
\draw[red] (5.corner 1)--(5.corner 2)  (5.corner 3)--(5.corner 4);

\draw[red] (6.corner 1)--(6.corner 4);
\draw[red] (7.corner 2)--(7.corner 3);
\draw[red] (8.corner 1)--(8.corner 3);
\draw[red] (9.corner 2)--(9.corner 4);
\draw[red] (10.corner 1)--(10.corner 2);
\draw[red] (11.corner 3)--(11.corner 4);

\draw[blue, densely dotted] (5*\r+2.5*\d,-\h) circle(\r cm);
\node[regular polygon, regular polygon sides=4, minimum size=2*\r cm] at (5*\r+2.5*\d,-\h) (-1) {};
\foreach \i in {1,...,4}
    \node[dot, fill=red] at (-1.corner \i) {};
\draw[red] (-1.corner 1)--(-1.corner 2)--(-1.corner 3)--(-1.corner 4)--cycle;

\draw[blue, densely dotted] (5*\r+2.5*\d, 2*\h) circle(\r cm);
\node[regular polygon, regular polygon sides=4, minimum size=2*\r cm] at (5*\r+2.5*\d,2*\h) (12) {};
\foreach \i in {1,...,4}
    \node[dot, fill=red] at (12.corner \i) {};
    
\foreach \i in {0,...,5}    {
\draw[blue, densely dotted] (5*\r+2.5*\d,-\h+\r)--(\i*2*\r+\i*\d,0-\r);
\draw[blue, densely dotted] (5*\r+2.5*\d,2*\h-\r)--(\i*2*\r+\i*\d,\h+\r);}

\draw[blue, densely dotted] (0*\r+0*\d,\h-\r)--(0*\r+0*\d,0+\r);
\draw[blue, densely dotted] (0*\r+0*\d,\h-\r)--(2*\r+1*\d,0+\r);
\draw[blue, densely dotted] (0*\r+0*\d,\h-\r)--(6*\r+3*\d,0+\r);

\draw[blue, densely dotted] (2*\r+1*\d,\h-\r)--(0*\r+0*\d,0+\r);
\draw[blue, densely dotted] (2*\r+1*\d,\h-\r)--(4*\r+2*\d,0+\r);
\draw[blue, densely dotted] (2*\r+1*\d,\h-\r)--(8*\r+4*\d,0+\r);

\draw[blue, densely dotted] (4*\r+2*\d,\h-\r)--(2*\r+1*\d,0+\r);
\draw[blue, densely dotted] (4*\r+2*\d,\h-\r)--(4*\r+2*\d,0+\r);

\draw[blue, densely dotted] (6*\r+3*\d,\h-\r)--(8*\r+4*\d,0+\r);
\draw[blue, densely dotted] (6*\r+3*\d,\h-\r)--(6*\r+3*\d,0+\r);

\draw[blue, densely dotted] (8*\r+4*\d,\h-\r)--(6*\r+3*\d,0+\r);
\draw[blue, densely dotted] (8*\r+4*\d,\h-\r)--(10*\r+5*\d,0+\r);
\draw[blue, densely dotted] (8*\r+4*\d,\h-\r)--(4*\r+2*\d,0+\r);

\draw[blue, densely dotted] (10*\r+5*\d,\h-\r)--(10*\r+5*\d,0+\r);
\draw[blue, densely dotted] (10*\r+5*\d,\h-\r)--(8*\r+4*\d,0+\r);
\draw[blue, densely dotted] (10*\r+5*\d,\h-\r)--(2*\r+1*\d,0+\r);
\end{tikzpicture}  
\caption{The lattice of noncrossing partitions of a square, that is, of the 4-element set. By Theorem \ref{iso-p-loc}, it corresponds to the lattice of all $p\,$-local localizing subcategories of $\mathcal{B}(\{1,2,3\})$.}  \label{fig:4}
\end{figure}
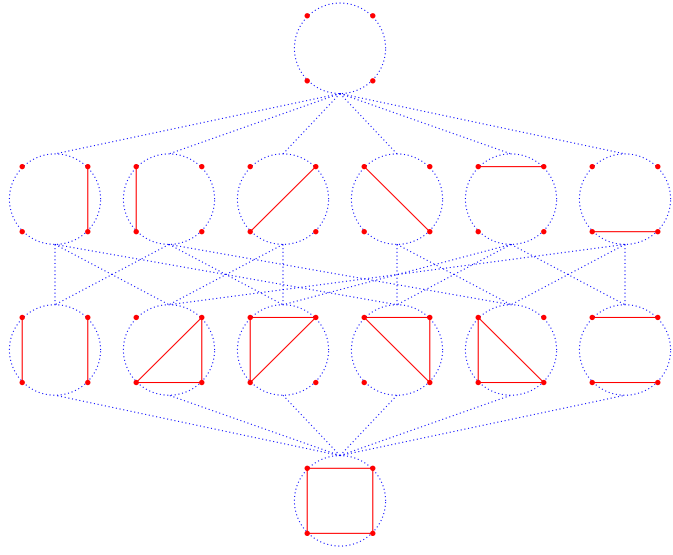

\begin{theorem} \label{iso-p-loc}
 There is a lattice isomorphism between $p$-local localizing subcategories of $\mathcal{B}(X)$ ordered by inclusion and $\mathrm{NC}_{n+1}$, the lattice of noncrossing partitions of a set with $n+1$ elements.
\end{theorem}

\begin{proof}
Denote the lattice of $p$-local localizing subcategories of $\mathcal{B}(X)$ by $L_n$. We are going to construct a lattice isomorphism $$\psi : L_n \stackrel{\sim}{\longrightarrow} \mathrm{NC}_{n+1}.$$  

By Theorem \ref{maintt}, a localizing subcategory $\mathcal{L} \in L_n$ is determined by the sets $\mathrm{U}^{\mathcal{L}}_{[a,b]}$ for $1\leq a \leq b\leq n.$ Given $\mathcal{L}$, we define a symmetric relation $\psi(\mathcal{L})$ on $\{1,\dots ,n+1\}$ by $a \sim b+1, \, b+1 \sim a \iff \mathrm{U}^{\mathcal{L}}_{[a,b]} = \emptyset$ for $a\leq b$ and $a\sim a$ for all $a \in \{1,\dots ,n+1\}$.

We want to show that $\psi(\mathcal{L})$ is indeed a noncrossing partition. First, we prove transitivity. Let $a,b,c \in \{1,\dots,n+1\}$ and $a\sim b, \; b\sim c$. If $a=b$ or $b=c$ or $a=c$ the assertion is trivial; so we assume they are all distinct. Define $x_1:=\min\{a,b,c\},\, x_3:=\max\{a,b,c\}$ and let $x_2$ be the remaining third point. Thus $x_1<x_2<x_3$.

In the proof of Lemma \ref{gen}, we showed that for $V \setminus W,W\setminus V, V \cup W  \in \LCX$ there is the following exact triangle in $\mathcal{B}(X)$:
$$ \Sigma \R_{W \setminus V} \rightarrow \R_{V} \rightarrow \R_{V \cup W} \rightarrow \R_{W \setminus V}.$$
Setting $V=[x_1,x_2-1]$ and $W=[x_1,x_3-1]$, and applying the functor $\mathrm{KK}_*(X;-,A)$ for any $A \in \mathcal{B}(X)$ to this triangle, we get the six term exact sequence 

\begin{displaymath}
    \xymatrix@C=1.2em{ \mathrm{KK}_0(X;\R_{[x_2,x_3-1]},A)  \ar[r]  & \mathrm{KK}_0(X;\R_{[x_1,x_3-1]} ,A) \ar[r] & \mathrm{KK}_0(X;\R_{[x_1,x_2-1]},A)  \ar[d] \\
               \mathrm{KK}_1(X;\R_{[x_1,x_2-1]},A)  \ar[u]  & \mathrm{KK}_1(X;\R_{[x_1,x_3-1]} ,A)  \ar[l] & \mathrm{KK}_1(X;\R_{[x_2,x_3-1]},A) \ar[l] }
\end{displaymath}

Theorem \ref{repres} gives $\mathrm{KK}_*(X;\R_Y,A) \cong \mathrm{FK}_Y(A)=\K_*(A(Y))$. Hence

$$
    \xymatrix{ \K_0(A([x_2,x_3-1]))  \ar[r]  & \K_0(A([x_1,x_3-1]))  \ar[r] & \K_0(A([x_1,x_2-1]))   \ar[d] \\
               \K_1(A([x_1,x_2-1]))   \ar[u]  & \K_1(A([x_1,x_3-1]))   \ar[l] & \K_1(A([x_2,x_3-1]))  \ar[l]  }
$$
The exactness of the latter sequence implies 
\begin{multline*}
\hspace*{-0.35cm}  \supp_\mathbb{Z} \K_*(A([x_2,x_3-1])) \subseteq  \supp_\mathbb{Z} \K_*(A([x_1,x_2-1])) \cup \supp_\mathbb{Z}\K_*(A([x_1,x_3-1])),\\
 \supp_\mathbb{Z} \K_*(A([x_1,x_3-1])) \subseteq  \supp_\mathbb{Z} \K_*(A([x_2,x_3-1])) \cup \supp_\mathbb{Z} \K_*(A([x_1,x_2-1])),\\
 \supp_\mathbb{Z} \K_*(A([x_1,x_2-1])) \subseteq  \supp_\mathbb{Z} \K_*(A([x_1,x_3-1])) \cup \supp_\mathbb{Z} \K_*(A([x_2,x_3-1])).
\end{multline*}
Therefore, by definition $\mathrm{U}^{\mathcal{L}}_{[x_2,x_3-1]} \subseteq \mathrm{U}^{\mathcal{L}}_{[x_1,x_2-1]} \cup \mathrm{U}^{\mathcal{L}}_{[x_1,x_3-1]}$, $\mathrm{U}^{\mathcal{L}}_{[x_1,x_3-1]} \subseteq \mathrm{U}^{\mathcal{L}}_{[x_2,x_3-1]}$ $\cup \mathrm{U}^{\mathcal{L}}_{[x_1,x_2-1]}$, $\mathrm{U}^{\mathcal{L}}_{[x_1,x_2-1]} \subseteq \mathrm{U}^{\mathcal{L}}_{[x_1,x_3-1]} \cup \mathrm{U}^{\mathcal{L}}_{[x_2,x_3-1]}$. So for any distinct $i, j =1,\dots,3$ such that $a=x_i$ and $c=x_j$, we get $\mathrm{U}^{\mathcal{L}}_{[a,c-1]} = \emptyset$; thus $a \sim c$, proving that $\psi(\mathcal{L})$ is a partition.

Now let $1\leq a<b<c<d \leq n+1$ and $a \sim c,\, b \sim d$; so $\mathrm{U}^{\mathcal{L}}_{[a,c-1]}=\emptyset$ and $\mathrm{U}^{\mathcal{L}}_{[b,d-1]} = \emptyset$.  In the proof of Lemma \ref{gen}, we also had a triangle

 $$\Sigma \R_W \rightarrow \R_V \rightarrow \R_{V \cup W}\oplus \R_{V \cap W} \rightarrow \R_W.$$ After setting $V=[a,c-1]$ and $W=[b,d-1]$, the same argument as above gives that $\mathrm{U}^{\mathcal{L}}_{[a,d-1]} \cup \mathrm{U}^{\mathcal{L}}_{[b,c-1]} \subseteq \mathrm{U}^{\mathcal{L}}_{[a,c-1]} \cup \mathrm{U}^{\mathcal{L}}_{[b,d-1]}$. Thus $a\sim b\sim c \sim d$. So the partition $\psi(\mathcal{L})$ is noncrossing.

 If $\psi(\mathcal{L})=\psi(\mathcal{L'})$, then $\mathrm{U}^{\mathcal{L}}_Y=\mathrm{U}^{\mathcal{L'}}_Y$ for all $Y \in \LCX$ since $\mathcal{L},\mathcal{L'}$ are $p\,$-local. So $\mathcal{L}=\mathcal{L'}$ by Theorem \ref{maintt}. So $\psi$ is injective.

\begin{figure}
\tikzset{dot/.style={circle,fill=red,minimum size=3pt, inner sep=0pt}}
\begin{tikzpicture}

\draw[blue, densely dotted] (0,0) circle(1.5cm);
\node[regular polygon, regular polygon sides=12, minimum size=3cm] at (0,0) (A) {};

\foreach \i in {1,...,12}
    \node[ dot, fill=red] at (A.corner \i) {};
   
\draw[thick,red] (A.corner 1)--(A.corner 2)--(A.corner 3)--(A.corner 4)--(A.corner 10)--(A.corner 11)--(A.corner 12)--cycle;
\draw[thick,red] (A.corner 5)--(A.corner 6)--(A.corner 7)--(A.corner 8)--(A.corner 9)--cycle;

    \node[anchor=(360/12)*1+180] at (A.corner 1) {1};
    \node[anchor=(360/12)*4.8+180] at (A.corner 4) {$a$};
    \node[anchor=(360/12)*10.8+180] at (A.corner 9) {$b+1$};
    
\draw[blue, densely dotted] (5,0) circle(1.5cm);
\node[regular polygon, regular polygon sides=12, minimum size=3cm] at (5,0) (B) {};

\foreach \i in {1,...,12}
    \node[ dot, fill=red] at (B.corner \i) {};
 
\draw[thick,red] (B.corner 4)--(B.corner 6)--(B.corner 12)--cycle;
\draw[thick,red] (B.corner 8)--(B.corner 9)--(B.corner 11)--cycle;
\draw[densely dashed,red] (B.corner 1)--(B.corner 2)--(B.corner 3)--(B.corner 4)--(B.corner 5)--(B.corner 6)--(B.corner 7)--(B.corner 12)--cycle;
 \draw[densely dashed,red] (B.corner 8)--(B.corner 9)--(B.corner 10)--(B.corner 11)--cycle;
    
    \node[anchor=(360/12)*1+180] at (B.corner 1) {1};
    \node[anchor=(360/12)*4.8+180] at (B.corner 4) {$a$};
    \node[anchor=(360/12)*10.8+180] at (B.corner 9) {$b+1$};

\end{tikzpicture}
\caption{The first picture shows the decomposition into two connected blocks corresponding to the interval $[a,b]$. The second picture is an example of a ``separating'' decomposition (indicated by dashed lines) for a noncrossing partition drawn with bold lines. } \label{fig:5}  
\end{figure}
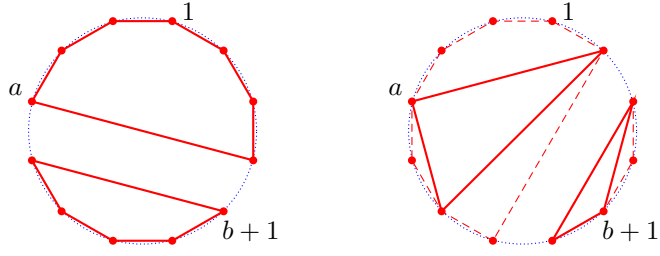
 
Now we prove that $\psi$ is surjective. The subintervals of $[1, n]$ are in one-to-one correspondence with the decompositions of the $n + 1$-gon into two (nonempty) connected subsets (see Figure~\ref{fig:5}). Here $[a,b] \in \LCX$ corresponds to the decomposition into $[a+1,b+1]$ and its complement. Let $\mathcal{L} = \langle \R_{[a,b]}^p \rangle$. If $1\leq x <y\leq n+1$, then~(\ref{conditions}) implies $x \nsim y$ in $\psi(\mathcal L)$ if and only if $x \leq a <y \leq b+1$ or $a<x \leq b+1 <y$. This exactly means that $\psi(\mathcal L)$ is a decomposition of $n + 1$-gon into two connected blocks (as in Figure~\ref{fig:5}). Since $\psi$ is injective, it gives a bijection between $p$-local localizing subcategories generated by a single interval and noncrossing partitions which are decompositions into two connected subsets.

Given a noncrossing partition $\sigma \in \mathrm{NC}_{n+1}$, $\psi^{-1}(\sigma)$ should be a $p\,$-local localizing subcategory $\mathcal{L}$ with $\mathrm{U}_{[a,b]}^{\mathcal{L}} = \emptyset$ if and only if $a \sim b+1$ in $\sigma$. For such an $\mathcal{L}$ to exist, we must show that this family of subsets $\mathrm{U}_{[a,b]}^{\mathcal{L}}$ satisfies the condition in Lemma~\ref{u}. This is trivially satisfied if $\mathrm{U}_{[a,b]}=\emptyset$. If $\mathrm{U}_{[a,b]}=\{p\}$ then $a$ and $b+1$ are in different blocks of $\sigma$. By the noncrossing property, one finds another decomposition into two connected subsets (corresponding by~$\psi$ to an interval $[c,d]$) such that it contains~$\sigma$, and $a$ and $b+1$ are in different blocks of the decomposition. We call these decompositions ``separating''. To construct this, for example, one could move from vertex $a$ on the $n+1$-gon clockwise and counterclockwise connecting all vertices to~$a$ along the way until the block of $b+1$ is reached, and connect all the remaining vertices to $b+1$ (see Figure~\ref{fig:5}).  This implies that $[a,b] \in B_{[c,d]}$. Moreover, if $x \nsim y+1$ in the decomposition corresponding to $[c,d]$ then $x \nsim y+1$ also in~$\sigma$. Hence if $[x,y] \in B_{[c,d]}$ then $\mathrm{U}_{[x,y]}=\{p\}$ for $\sigma$. In other words, the localization condition is satisfied. Therefore, $\psi$ is a bijection.

It is straightforward to see that $\psi$ and $\psi^{-1}$ are inclusion and refinement preserving, respectively. Therefore, $\psi$ is an isomorphism of lattices.
  \end{proof}
As a corollary, we get our main result
\begin{theorem}
\label{maint}
 The lattice of localizing subcategories of $\mathcal{B}(X)$ is isomorphic to $\prod_{p \in \mathrm{Spec }\mathbb{Z}}\mathrm{NC}_{n+1}^p$.
\end{theorem}
\begin{proof}
 The statement directly follows from Theorem \ref{iso-p-loc} and Remark \ref{p-loc}.
\end{proof}

\bibliography{bib}
\bibliographystyle{plain}
 \end{document}